\documentclass[a4paper,12pt]{article}

\usepackage{amsthm,amsmath,stmaryrd,bbm,hyperref,geometry,color,authblk}
\usepackage[utf8]{inputenc}
\usepackage{amssymb}
\usepackage[english,francais]{babel}
\usepackage{graphicx}
\usepackage{amsfonts,amssymb}
\usepackage{verbatim}
\usepackage{enumitem}

\newcommand{\po}{\left(}
\newcommand{\pf}{\right)}
\newcommand{\co}{\left[}
\newcommand{\cf}{\right]}
\newcommand{\cco}{\llbracket}
\newcommand{\ccf}{\rrbracket}
\newcommand{\R}{\mathbb R} 
\newcommand{\T}{\mathbb T} 
\newcommand{\Z}{\mathbb Z} 
\newcommand{\N}{\mathbb N} 
\newcommand{\X}{\mathbb X}
\newcommand{\bfV}{\mathbf V} 
\newcommand{\dd}{\text{d}}
\newcommand{\na}{\nabla}
\newcommand{\1}{\mathbbm{1}}

\newcommand{\rmd}{\mathrm{d}}
\newcommand{\bfz}{\mathbf{z}}
\newcommand{\bfx}{\mathbf{x}}
\newcommand{\bfX}{\mathbf{X}}
\newcommand{\bfy}{\mathbf{y}}
\newcommand{\bfv}{\mathbf{v}}

\newcommand{\TV}{\mathrm{TV}}
\newcommand{\mub}{\bar{\mu}_*}

\newcommand{\bleu}[1]{\textcolor{blue}{#1}}
\newcommand{\rouge}[1]{\textcolor{red}{#1}}
\newcommand{\magenta}[1]{\textcolor{magenta}{#1}}

\newtheorem{thm}{Theorem}
\newtheorem{assu}{Assumption}

\newtheorem{lem}[thm]{Lemma}

\newtheorem{cor}[thm]{Corollary}
\newtheorem{prop}[thm]{Proposition}

\title{Non-asymptotic entropic  bounds for  non-linear kinetic Langevin sampler with second-order  splitting scheme}
\author[1,2]{Pierre Monmarché}
\author[3]{Katharina Schuh}
\affil[1]{Sorbonne Université,  Laboratoire Jacques-Louis Lions \& Laboratoire de Chimie théorique, LJLL \& LCT, F-75005 Paris}
\affil[2]{Institut Universitaire de France}
\affil[3]{TU Wien, Institute of Analysis and Scientific Computing, Wiedner Hauptstraße 8–10, 1040 Wien}

\begin{document}
 \selectlanguage{english}
 
\maketitle

\begin{abstract}
    The problem of sampling according to the probability distribution minimizing a given free energy, using interacting particles unadjusted kinetic Langevin Monte Carlo, is addressed. In this setting, three sources of error arise, related to three parameters: the number of particles $N$, the discretization step size $h$, and the length of the trajectory $n$. The main result of the present work is a quantitative estimate of strong convergence in relative entropy, implying non-asymptotic bounds for the quadratic risk of Monte Carlo estimators for bounded observables. The numerical discretization scheme considered here is a second-order splitting method, as commonly used in practice. In addition to $N,h,n$, the dependency in the ambient dimension $d$ of the problem is also made explicit, under suitable conditions. The main results are proven under general conditions (regularity, moments, log-Sobolev inequality), for which tractable conditions are then provided. In particular, a Lyapunov analysis is conducted under more general conditions than previous works; the nonlinearity may not be small and it may not be convex along linear interpolations between measures.
\end{abstract}
 
\section{Overview}\label{sec:overview}

Denoting by $\mathcal P_2(\X)$ the set of probability measures with finite second moment over a set $\X$ (either $\X =\R^d$ or $\X=\mathbb T^d$ with the flat torus $\T=\R/\Z$), consider an energy functional $F:\mathcal P_2(\X) \rightarrow \R$, and the associated free energy
\[\mathcal F(\mu) = F(\mu) +   H(\mu)\]
for $\mu\in\mathcal P_2(\X)$, where $H(\mu)=\int \mu \ln\mu$  stands for the Boltzmann entropy (taken as $+\infty$ if $\mu$ does not have a Lebesgue density). The conditions introduced below will ensure that  $\mathcal F$ has a unique global minimizer $\bar{\mu}_*$, which is known to solve
the self-consistency equation
\begin{equation}
    \label{eq:mu*self-consistent}
    \bar{\mu}_*  \propto \exp \po - U_{\bar{\mu}_*}(x)\pf  \dd x\,,
\end{equation} 
where, for $\mu \in \mathcal P_2(\X)$, we write $U_\mu$ for the linear derivative of $F$ (see \eqref{eq:defflatderiv}) at $\mu$,
\[U_{\mu}(x) =   \frac{\delta F}{\delta m}(\mu,x)\,.\]
As an example, a classical case is given by the energy
\begin{equation}
    \label{eq:Fmu_potentials}
    F(\mu) = \int_{\X} V(x) \mu(\dd x) +  \frac12 \int_{\X^2} W(x,y) \mu(\dd x)\mu(\dd y)
\end{equation}
for some external and interaction potentials $V,W$ with $W(x,y)=W(y,x)$, for which
\[U_{\mu}(x) = V(x) +  \int_{\X} W(x,y) \mu(\dd y) \,. \]

We are concerned with the question of sampling $\bar{\mu}_*$. This is classically done by approximating it by the marginal equilibrium distribution of a system of $N$ mean-field interacting particles sampling the measure
\begin{equation}
    \label{def:muNUN}
    \mu_\infty^N(\bfx) \propto \exp(-U_N(\bfx))\,,\qquad U_N(\bfx) = N F(\pi_{\bfx})\,,\qquad \pi_{\bfx} = \frac1N\sum_{i=1}^N\delta_{x_i}\,,
\end{equation}
where we write $\bfx = (x_1,\dots,x_n)\in\X^N$. Under suitable conditions, the particles are approximately independent (this is the so-called propagation of chaos phenomenon), their empirical distribution $\pi_{\bfx}$ approximates their common marginal law, which approximates $\bar{\mu}_*(x)$. 

The measure $\mu_\infty^N$ being known up to its normalizing constant, can be sampled using any Markov Chain Monte Carlo (MCMC) method. In this work, as motivated in the next paragraph, we consider an unadjusted kinetic Langevin chain, obtained by applying a splitting discretization scheme to the kinetic Langevin diffusion, which can also be seen as a particular case of generalized Hamiltonian Monte Carlo. The method then suffers from three sources of errors, the particle approximation depending on $N$, the discretization error depending on the step-size $h$, and the long-time convergence to stationarity depending on the physical trajectory length $hn$ where $n$ stands for the number of transitions of the chain. Our goal is to obtain non-asymptotic error bounds between the law of the chain and the target measure in terms of these three parameters.

Since there has been a lot of activity on obtaining non-asymptotic complexity bounds for MCMC methods recently (as discussed in more details in Section~\ref{sec:relatedworks}), let us already clarify the scope and contributions of this work. First, among all possible samplers, the focus on kinetic Langevin is motivated by 1) its non-reversible ballistic behavior which makes it efficient in particular for ill-conditioned targets~\cite{Gouraud_etal}, 2) the second-order accuracy of the associated splitting scheme for smooth potentials \cite{Leimkuhler} and 3) its dominant use in some domains, in particular Molecular Dynamics simulations \cite{lelievre2016partial}. Moreover, we are interested in cases where $U_N$ is not convex. Closely related works, concerned with convergence bounds for mean-field kinetic samplers in non-convex settings,  are thus \cite{BouRabeeSchuh,Camrudetal,Chen_etal,GuillinMonmarche,fu2023mean}. Let us discuss the present work in light of these previous studies. First, \cite{BouRabeeSchuh,Camrudetal,GuillinMonmarche} only consider the pair interaction case~\eqref{eq:Fmu_potentials}, while we consider the present more general setting as in \cite{Chen_etal,fu2023mean} which is motivated by the important activity over the recent years on the mean-field analysis in machine learning and statistics \cite{Chizat,Szpruch,mei2018mean,mei2019mean,nitanda2022convex}. Second, \cite{Chen_etal,GuillinMonmarche} consider the continuous-time process, not a practical scheme. Third, \cite{BouRabeeSchuh,Camrudetal,GuillinMonmarche} require a small non-linearity, and by contrast \cite{Chen_etal,fu2023mean} requires $F$ to be flat-convex, while our assumptions are more general than both these settings (and in particular cover simultaneously the flat-convex and small interaction cases; see an example in Section~\ref{sec:exnonconvex}). Next, \cite{BouRabeeSchuh} (which, besides, rather than the kinetic Langevin, consider the classical unadjusted HMC sampler, which has a diffusive behavior for ill-conditioned targets \cite{Gouraud_etal}) is based on direct coupling methods and thus yields a convergence result in terms of Wasserstein $1$ distance while, as in \cite{Camrudetal,Chen_etal,GuillinMonmarche,fu2023mean}, we work with entropy methods, obtaining stronger convergence in relative entropy and total variation (and also in some cases sharper rates of convergence as discussed in \cite[Remark 6]{Monmarcheidealized}).

As a summary, the present work somehow combines and extends \cite{Camrudetal} and \cite{Chen_etal}, following a similar entropy method, considering both the discretization error (contrary to \cite{Chen_etal}) and the propagation of chaos error (contrary to \cite{Camrudetal}), while relaxing the conditions of both works. In particular, apart from the main entropy dissipation argument, the method of \cite{Camrudetal} requires some uniform moment bounds to control the discretization error, which are obtained classically by designing a suitable Lyapunov function and, in \cite{Camrudetal}, are only established when the non-linearity is small, while in the present work we consider conditions which allow for instance  arbitrarily large bounded non-linearities (as there is no smallness restriction on the constant $M$ in \eqref{condition:lambdaM} below). Finally, building upon \cite{SongboLSI} (and its extension \cite{MonmarcheLSI}) which noticed that the method of~\cite{Chen_etal,Chen1} (where \cite{Chen1} by the same authors is similar to \cite{Chen_etal} but for the reversible overdamped Langevin diffusion) can be cast in terms of defective log-Sobolev inequality (LSI), our treatment of the $N$ particle approximation is arguably clarified with respect to~\cite{Chen_etal}. Indeed, the entropy dissipation (Proposition~\ref{prop:Camrud} below), the LSI (Assumption~\ref{assu:LSI}) and the entropy comparison~\eqref{eq:compareEntropy} are treated separately, instead of merged in a single approximate entropy dissipation inequality as in the proof of \cite[Theorem 2.2]{Chen_etal}. 

Taking into account both errors in terms of $N$ and $h$ is also done in \cite{fu2023mean}, which differs from our work as follows: it only considers cases where $F$ is flat-convex, uses a first-order Euler schemes, and finally the long-time convergence/propagation of chaos is only proven for the continuous-time process (as in \cite{Chen_etal}) and then the numerical discretization error is treated separately, not uniformly in time (leading to a non-sharp dependency on the log-Sobolev constant in the step-size and number of iterations, see \cite[Equations (70) and (71)]{fu2023mean}), and requires a triangular inequality to conclude so that the result holds in total variation (through Pinsker's inequality) but not for the relative entropy itself. 

Finally, let us mention \cite{Suzukietal}, which is very similar to our work except that it is applied to the overdamped Langevin dynamics (extending \cite{VempalaWibisono} to the mean-field settings, where by comparison our work is based on \cite{Camrudetal}), leading to a first-order scheme. Also, again, there, $F$ is assumed to be flat-convex.

Our main general result is Theorem~\ref{thm:GlobalEntropyDecay}, which gives, under  abstract conditions, a non-asymptotic quantitative relative entropy bound between the law of the output of the algorithm and the measure $\mu_\infty^N$. Combined with Proposition~\ref{prop:quadrisk}, this implies error bounds in terms of quadratic risk for Monte Carlo estimators of expectations of bounded functions with respect to $\mub$. We also state some results which allows to check the general assumptions of Theorem~\ref{thm:GlobalEntropyDecay}, in particular Propositions~\ref{prop:LyapTorus} and \ref{prop:LyapunovRd} which establish Lyapunov conditions for the numerical scheme, implying uniform-in-time moment bounds.

 \bigskip

 The rest of this work is organized as follows. Section~\ref{sec:settings} contains the main general definitions, assumptions and results. Some tractable conditions to check the assumptions of the main results are presented in Section~\ref{sec:discussion}, with Section~\ref{sec:conditionsLSI&conditionEntropy} devoted to some entropy comparison and log-Sobolev inequality, while Section~\ref{sec:conditionLyap} adresses Lyapunov conditions. This discussion is then concluded by some bibliographical references in Section~\ref{sec:relatedworks}. Some examples of applications are presented in Section~\ref{sec:examples}. The proofs are given in Section~\ref{sec:mainproof} for the general results, in Section~\ref{sec:proofLyap} for the Lyapunov conditions and in Appendix~\ref{sec:auxiliaryproof} for an auxiliary lemma.

\section{Settings and results}\label{sec:settings}

\subsection{The chain}

Given a step-size $h>0$, we denote by $\Phi(\bfx,\bfv)=(\bar{\bfx},\bar{\bfv})$ the result of one step of Verlet integrator of the Hamiltonian dynamics associated to $U_N$ starting from $\bfz=(\bfx,\bfv) \in\X^N \times \R^{dN}$, given by
\begin{eqnarray}
\label{Verlet}
    \bar{\bfx} &= & \bfx + h \bfv - \frac{h^2}{2} \na U_{N}(\bfx) \\
    \bar{\bfv} &= & \bfv - \frac{h}{2}\po \na U_{N}(\bfx) + \na U_{N}(\bar{\bfx}) \pf\,.
\end{eqnarray}

For a friction parameter $\gamma>0$, assuming that $h<\gamma^{-1}$ we set 
\begin{equation}
    \label{def:eta}
\eta=1-\gamma h
\end{equation}
 and write $\mathcal D$ the Markov operator associated to a partial Gaussian refreshment of the velocity with damping parameter $\eta$,
\[(\bfx,\bfv) \rightarrow (\bfx,\eta \bfv + \sqrt{1-\eta^2} \mathbf{G})\,,\qquad \mathbf{G}\sim \mathcal N(0,I_{dN})\,, \]
namely
\[\mathcal D f(\bfx,\bfv) = \mathbb E \po f\po \bfx,\eta \bfv + \sqrt{1-\eta^2} \mathbf{G}\pf \pf \,. \]
The subject of the present work is the Markov chain associated to the operator $\mathcal P$ given by
\[\mathcal P f(\bfx,\bfv) = \mathcal D \po f \circ \Phi\pf (\bfx,\bfv)\,, \]
corresponding to, successively, one step of velocity randomization  followed by one step of Verlet.

We write $\nu_0^N \in \mathcal P_2(\X^N\times\R^{dN})$ the initial distribution of the chain,
\[\nu_n^N = \nu_0^N \mathcal P^n\]
the law after $n$ steps, and $\mu_n^N \in \mathcal P_2(\X^N)$ the position marginal of $\nu_n^N$.

We assume that $\exp(-U_N)$ is integrable  for all $N\in\N$, so that the probability measure $\mu_\infty^N$ given in~\eqref{def:muNUN} is well-defined, and we set
\[\nu_\infty^N = \mu_\infty^N \otimes \mathcal N(0,I_{dN})\,. \]
This measure would be invariant for $\mathcal P$ if the Hamiltonian $H(\bfx,\bfv) = U_N(\bfx) + \frac12 |\bfv|^2$ were exactly preserved by the Hamiltonian case, which is not the case. However, when the step-size $h$ is small, $\mathcal P$ is expected to have a unique invariant measure, close to $\nu_\infty^N$.

\subsection{General assumptions}

We start with regularity conditions on $F$. The linear derivative of a functional $F:\mathcal P_2(\X)\rightarrow \R$ is the function $\frac{\delta F}{\delta m} :\mathcal P_2(\X)\times\X \rightarrow \R$ such that
\begin{equation}
    \label{eq:defflatderiv}
    F(\mu) - F(\nu) = \int_0^1 \int_{\X}\frac{\delta F}{\delta m}((1-t)\nu + t\mu,x) (\mu-\nu)(\dd x)\dd t
\end{equation}
for $\mu,\nu\in\mathcal P_2(\X)$. We write $DF(\mu,x)=\na_x \frac{\delta F}{\delta m}(\mu,x)$, called the intrinsic derivative of $F$. When $U_N$ is given by~\eqref{def:muNUN}, 
\begin{equation}
    \label{eq:nablaUN}
\na_{x_i} U_N(\bfx) = DF(\pi_{\bfx},x_i) = \na U_{\pi_\bfx}(x_i)\,.
\end{equation}
The second order linear derivative of $F$ is the function such that $\mu,x'\mapsto \frac{\delta^2 F}{\delta m^2}(\mu,x,x')$ is the linear derivative of $\mu\mapsto \frac{\delta F}{\delta m}(\mu,x)$ for a fixed $x\in\X$. Write $D^2 F(\mu,x,x')= \na_{x,x'}^2 \frac{\delta^2 F}{\delta m^2}(\mu,x,x')$ the second order intrinsic derivative of $F$. Then
\begin{equation} \label{eq:nabla2UN}
\begin{aligned}
\na^2_{x_i,x_j} U_N(\bfx) &= \frac1N D^2 F(\pi_{\bfx},x_i,x_j) + \1_{i=j} \na_{x_i} DF(\pi_{\bfx},x_i)  \\
&=  \frac1N D^2 F(\pi_{\bfx},x_i,x_j) + \1_{i=j} \na^2 U_{\pi_\bfx}(x_i)\,.
\end{aligned}
\end{equation}
The higher order linear and intrinsic derivatives of $F$ are defined similarly. 

\begin{assu}
    \label{assu:regularity}
    The intrinsic derivatives of $F$ of order $1$ to $4$ exist. Moreover, $x\mapsto DF(\mu,x)$ (resp. $(x,x')\mapsto D^2 F(\mu,x,x')$, resp. $(x,x',x'') \mapsto D^3 F(\mu,x,x',x'')$) is $\mathcal C^3$ (resp. $\mathcal C^2$, resp. $\mathcal C^1$) with derivative of order 1 to  3  (resp. $1$ to $2$, resp. $1$) bounded uniformly in $\mu \in\mathcal P_2(\X)$. Finally, $\mu \mapsto DF(\mu,x)$
    is Lipschitz continuous with respect to the $\mathcal W_2$ Wasserstein distance. 
    In particular, there exist $M_{1,m},M_{2,m}>0$ such that 
    \[\forall x,x'\in\X,\ \mu,\mu'\in\mathcal P_2(\X),\quad |D F(\mu,x)  - D F(\mu',x')| \leqslant M_{1,x}|x-x'| + M_{1,m} \mathcal W_2(\mu,\mu')\,. \]
    
\end{assu}

The relation with respect to the regularity conditions of \cite{Camrudetal} is the following, proven in Appendix~\ref{sec:auxiliaryproof}:
\begin{lem}\label{lem:regularity}
    Under Assumption~\ref{assu:regularity},  $\na U_N$ is $L_1$-Lipschitz continuous for all $N\in\N$ with $L_1 = M_{1,x}+M_{1,m}$ and there  exist $L_2,L_3>0$  such that, for any $N\in\N$, $\bfx\in\X^N$ and $\bfy,\bfz\in\R^{dN}$,
 \begin{align}
\left| \po  \na^2 U_N(\bfx+\bfy) - \na^2 U_N(\bfx)\pf \bfz\right| & \leqslant      L_2\|\bfy\|_4\|\bfz\|_4 \label{eqdef:assuU3bounded}\\
\left|\na U_N(\bfx+\bfy) - \na U_N(\bfx) - \frac12\po \na^2 U_N(\bfx)+\na^2 U_N(\bfx+\bfy)\pf \bfy\right|  &\leqslant    L_3\|\bfy\|_6^3
\, \label{eqdef:assuU4bounded}
\end{align}
where we write $\|\bfy\|_p = \po \sum_{i=1}^N |y_i|^p\pf^{\frac1p} $ for $p\in\N$.
\end{lem}

Moreover, from~\eqref{eq:nablaUN}, under Assumption~\ref{assu:regularity},
\begin{align}
|\na_{x_1} U_N(\bfx)|^2 & \leqslant 3|DF(\delta_0,0)|^2 + 3M_{1,x}^2 |x_1|^2 + 3 M_{1,m}^2 \mathcal W_2^2(\delta_0,\pi(\bfx))\nonumber \\
& = 3|DF(\delta_0,0)|^2 + 3M_{1,x}^2 |x_1|^2 + 3\frac{M_{1,m}^2  }{N} \sum_{i=1}^N |x_i|^2\,.\label{eq:lipbound1}
\end{align}



The second condition, also from~\cite{Camrudetal}, is uniform in time moment bounds. Deferring  to Section~\ref{sec:conditionLyap} for more explicit conditions under which this can be established, for now we state it as an assumption:

\begin{assu}
    \label{assu:moments}
    Under $\nu_0^N$, the $N$ particles are indistinguishable. Moreover,
\begin{equation}
    \label{eq:momentbound} 
\sup_{N\in\N} \sup_{n\in\N}  \int_{\X^N\times \R^{dN}}\po |x_1|^6 + |v_1|^6\pf \nu_n^N (\dd \bfx\dd \bfv) <\infty\,.
\end{equation}
\end{assu}

Here, when $\X=\T^d$, $|x_1|$ is understood as the norm of the represent of $x_1$ in $[0,1]^d$ (in particular in that case $|x_1|$ is in fact bounded and thus~\eqref{eq:momentbound}  is only a condition on the moments of the velocity variable).

Under both Assumptions~\ref{assu:regularity} and \ref{assu:moments}, thanks to~\eqref{eq:lipbound1} and the indistinguishability of particles, we get that
\begin{equation}
  C_1 :=   \sup_{N\in\N} \sup_{n\in\N}  \frac{1}{N d^3}\int_{\X^N\times \R^{dN}} \sum_{i=1}^3 L_i^2 \po \| \bfv \|_{2i}^{2i} + \|\na U_N(\bfx)\|_{2i}^{2i}\pf \nu_n (\dd \bfx\dd \bfv)  <\infty\,, \label{eq:C1}
\end{equation}
where $L_1,L_2,L_3$ have been introduced in Lemma~\ref{lem:regularity}. The normalisation with $d^3$ in the definition of $C_1$ is motivated by the fact that under the target distribution $\nu_\infty^N$, the expectation of $|v_1|^2 + |v_1|^4 +|v_1|^6$ is of order $d^3$. In other words, if the bounds on the derivatives in Assumption~\ref{assu:regularity} are thought as independent from $d$, then so can be $C_1$. See Section~\ref{sec:conditionLyap}, where Assumption~\ref{assu:moments} is established under some conditions on $F$, for further discussion on the dependency in $d$.

The third condition is a (possibly defective) log-Sobolev inequality (LSI) with some uniformity in $N$:

\begin{assu}
    \label{assu:LSI}
    There exists $\rho>0$ such that for all $N\in\N$, there exists $\delta_N>0$ such that $\mu_\infty^N$ given by~\eqref{def:muNUN} satisfies
    \begin{equation}
        \label{eq:defectiveLSI}
        \forall \mu \in \mathcal P_2(\X^N),\qquad \rho \mathcal H(\mu|\mu_\infty^N) \leqslant  \mathcal I(\mu|\mu_\infty^N) + \delta_N\,.
    \end{equation}
    Moreover, $\delta_N = o(N)$ as $N\rightarrow \infty$.
\end{assu}
We will not really use the fact that $\delta_N = o(N)$ for proving our main results, but anyway without this condition the bounds we obtain have no interest.

Conditions to establish this are discussed in Section~\ref{sec:conditionsLSI&conditionEntropy}. In some cases (using e.g. the results of \cite{GuillinWuZhang,SongboLSI,MonmarcheLSI}), the LSI is tight, meaning that $\delta_N=0$ (it is said defective otherwise), but  we will see some cases of interest where we get $\delta_N$  of order  $N^{\theta}$ with $\theta<1$ (see Section~\ref{sec:conditionsLSI&conditionEntropy}).

In particular, applying~\eqref{eq:defectiveLSI} with $\mu = \mu_1^{\otimes N}$ for some $\mu_1 \in \mathcal P_2(\X)$, dividing by $N$ and letting $N\rightarrow \infty$ following \cite{GuillinWuZhang}, we obtain a global  non-linear LSI (with the vocabulary of \cite{MonmarcheReygner}), which among other consequences imply the following:

\begin{lem}\label{eq:SelfConsist=>GlobalMin}
    Under Assumption~\ref{assu:LSI}, a probability measure $\mu$ solving the self-consistency equation $\mu \propto \exp(-U_{\mu})$ is necessarily a global minimizer of the free energy $\mathcal F$.
\end{lem}

The argument is given in \cite[Section 2.2]{MonmarcheReygner} to which we refer for details (in fact we won't use  Lemma~\ref{eq:SelfConsist=>GlobalMin} in our analysis, we only mention it to highlight that, conversely, Assumption~\ref{assu:LSI} cannot hold in cases with other solutions to the self-consistency equation than global minimizers, as studied in \cite{MonmarcheReygner}, for which global convergence is not expected).

\subsection{Main results}

Under the previous  conditions, we can generalize the result of \cite{Camrudetal} as follows.

\begin{thm}
    \label{thm:GlobalEntropyDecay}
    Under Assumptions~\ref{assu:regularity}, \ref{assu:moments} and \ref{assu:LSI}, assume furthermore that $h\sqrt{M_{1,x}+M_{1,m}}\leqslant 1/10$  and set
    \begin{equation}\label{eq:a}
     a = \frac{\gamma}{7+3(\gamma+3)^2}\,,\qquad \kappa = \frac{a}{3  \max(1,\rho^{-1})  + 6a }\,,\qquad C_2 =  \frac1\kappa \po 9 + \frac1a\pf  C_1\,,
    \end{equation}
    with $C_1$ given by~\eqref{eq:C1}.     Then, for all $N,n\in\N$,
    \[\mathcal H(\nu_n^N | \nu_\infty^N) \leqslant \po 1+ \kappa h\pf^{-n} \co \mathcal H(\nu_0^N | \nu_\infty^N) + 2 a \mathcal I(\nu_0^N | \nu_\infty^N) \cf  +  \rho^{-1} \delta_ N  + C_2 N d^3 h^4 \,. \]
\end{thm}

Since $\mathcal H(\mu_n^N | \mu_\infty^N) \leqslant \mathcal H(\nu_n^N | \nu_\infty^N)$, the quadratic risk of an estimator based on $N$ particles   $(X_1,\dots,X_N) \sim \mu_n^N$, defined as
\[ \mathfrak{R}(f,n,N,h) = \mathbb E \co \po \frac1N \sum_{i=1}^N f(X_i) - \bar{\mu}_*(f) \pf^2 \cf\,, \]
is then controlled by combining Theorem~\ref{thm:GlobalEntropyDecay} with the following simple result.

\begin{prop}\label{prop:quadrisk}
Let $f:\R^d\rightarrow \R$ be a bounded measurable function and, for $N,n\in\N$. Then,
\begin{equation}\label{eq:quadrisk1}
      \mathfrak{R}(f,n,N,h)  \leqslant 4 \|f\|_{\infty}^2 \po \frac1N   +   \sqrt{2\mathcal H(\mu_n^N |  \mu_\infty^N)} + \TV (\mu_\infty^{2,N},\mub^{\otimes 2})  \pf  \, .
\end{equation}
Alternatively, if we assume that there exists $R>0$ and $(\eta_N)_{N\in\N}\in\R_+^{\N}$ such that, for all $N\in\N$ and $\mu \in \mathcal P_2(\X^N)$,
\begin{equation}
    \label{eq:compareEntropy}
\mathcal H\po \mu|\bar{\mu}_*^{\otimes N}\pf \leqslant R \mathcal H\po \mu|\mu_\infty^N\pf + \eta_N\,, 
\end{equation}
  then
    \begin{equation}\label{eq:quadrisk2}
    \mathfrak{R}(f,n,N,h)  \leqslant   4 \|f\|_{\infty}^2 \po \frac1N  +   2\sqrt{\frac{\eta_N + R    \mathcal H(\mu_n^N | \mu_\infty^N)}{N} } \pf  \, . 
\end{equation}
\end{prop}

Conditions under which~\eqref{eq:compareEntropy} holds are discussed in Section~\ref{sec:conditionsLSI&conditionEntropy}.

To fix some ideas on the relative efficiency of the two bounds~\eqref{eq:quadrisk1} and \eqref{eq:quadrisk2}, let us first briefly discuss them when $n$ is large and $\delta_N=0$ (notice that \eqref{eq:quadrisk1} is only useful if $\delta_N \rightarrow 0$, while \eqref{eq:quadrisk2} is useful as soon as $\delta_N = o(N)$). In that case, considering only the dependency in $N$ and $h$, we get that $\mathcal H(\mu_n^N|\nu_\infty^N)$ is of order $N h^4$. This means that~\eqref{eq:quadrisk1} gives, for $n$ large, 
\begin{equation}
    \label{eq:complexityR1}
    \mathfrak{R}(f,n,N,h) =\mathcal O \po \frac1N + h^2\sqrt{N} +  \TV (\mu_\infty^{2,N},\mub^{\otimes 2})  \pf\,.
\end{equation}
Standard global propagation of chaos results at stationarity give a bound on $\TV (\mu_\infty^{2,N},\mub^{\otimes 2})$ of order $N^{-1/2}$. However, since the recent seminal work of Lacker~\cite{Lacker} and subsequent studies~\cite{LackerLeFlem,MonmarcheRenWang},  sharp bounds of order $N^{-1}$ are known in some situations. In the latter cases, the propagation of chaos error in~\eqref{eq:complexityR1} recovers the optimal Monte Carlo rate $N^{-1}$ (which hold for independent random variables).  However, to deal with the contribution of the discretization error in this bound, we need to take $h$ of order $N^{-1/2}$.

By contrast, if \eqref{eq:compareEntropy} holds, taking $\mu=\mu_\infty^N$ shows that $\eta_N \geqslant \mathcal H(\mu_\infty^N | \mub^{\otimes N} )$, which is typically of order $1$ (see \cite{GuillinWuZhang}).  When $\eta_N = \mathcal O(1)$, \eqref{eq:quadrisk2} gives
\begin{equation}
    \label{eq:complexityR2}
    \mathfrak{R}(f,n,N,h) =\mathcal O \po \frac1{\sqrt{N}} + h^2  \pf\,,
\end{equation}
which does not recover the sharp Monte Carlo rate. On the other hand, when the only known bounds on $\TV (\mu_\infty^{2,N},\mub^{\otimes 2})  $ are of order $N^{-1/2}$,  \eqref{eq:complexityR2} is always better than \eqref{eq:complexityR1}. Similarly, when $\delta_N$ is not $0$ but $N^{-c}$ for some $c\in(0,1/2]$ (as in Section~\ref{sec:conditionsLSI&conditionEntropy}), then the $N^{-1}$ rate is lost in both bounds and then~\eqref{eq:complexityR1} is always worse than~\eqref{eq:complexityR2}.  Besides, optimizing the right hand side of \eqref{eq:complexityR2} leads to a choice of $h$ independent from $N$.

Now, under  conditions similar to those where uniform-in-time sharp propagation of chaos has been shown for the continuous-time overdamped Langevin process~\cite{LackerLeFlem,MonmarcheRenWang}, we may expect  an analogue result to hold in our discrete-time kinetic case, resulting to a bound of order 
\begin{equation}
    \label{eq:complexityR3}
    \mathfrak{R}(f,n,N,h) =\mathcal O \po \frac1{N} + h^2  \pf\,,
\end{equation}
improving upon both~\eqref{eq:complexityR1} and \eqref{eq:complexityR2}. Establishing this is  difficult, as for now the works~\cite{LackerLeFlem,MonmarcheRenWang} require restrictive conditions on the non-linearity and temperature and consider only the overdamped Langevin process. Extending these arguments to the continuous-time kinetic Langevin process would be a first step in direction of~\eqref{eq:complexityR2} (hopefully, up to technicalities, the time discretization should then not be so difficult to deal with, following~\cite{Camrudetal}). 

Notice that in \eqref{eq:complexityR1}, we only need sharp propagation of chaos at stationarity, which is simpler than uniform-in-time bounds and is the same for the kinetic and overdamped processes.

To conclude for now, let us discuss complexity bounds obtained either with~\eqref{eq:complexityR1} (assuming sharp propagation of chaos at stationarity) or~\eqref{eq:complexityR2}. For a given $\varepsilon>0$, up to some logarithmic terms, the number $n$ of steps of the chain to achieve an error $\mathfrak{R}(f,n,N,h) \leqslant \varepsilon$ is of order $h^{-1}$. On the other hand, at each step, $\na_N U$ has to be computed, which has a cost of order $N^2$. So, the numerical cost of the algorithm is roughly $\mathcal O(N^2/h)$. Using~\eqref{eq:complexityR1} leads to the choice $N = \varepsilon^{-1}$, $h= \sqrt{\varepsilon}/N^{1/4}= \varepsilon^{3/4}$, hence a total cost $\mathcal O(\varepsilon^{-11/4})$. In contrast, using~\eqref{eq:complexityR2} leads to $N = \varepsilon^{-2}$, $h=\sqrt{\varepsilon}$, hence a cost $\mathcal O(\varepsilon^{-9/2})$. In particular, in this situation,~\eqref{eq:complexityR2}  is worse than~\eqref{eq:complexityR1}. If we had~\eqref{eq:complexityR3}, we would get a cost $\mathcal O(\varepsilon^{-5/2})$. 

\section{Discussion}\label{sec:discussion}

\subsection{Conditions for the defective LSI and entropy comparison}\label{sec:conditionsLSI&conditionEntropy}

In this section we discuss how to establish the LSI~\eqref{eq:defectiveLSI} and the entropy comparison~\eqref{eq:compareEntropy} in practical cases, building upon \cite{SongboLSI,MonmarcheLSI,Chen_etal}. The standing assumption here is that there exists $\mathcal C :\mathcal P_2(\X) \times \mathcal P_2(\X) \rightarrow \R$ such that 
\begin{equation}
\label{eq:courbure-cost}
\forall \mu,\nu\in\mathcal P_2(\X),\ t\in[0,1],\qquad F(t\mu + (1-t)\nu) \leqslant t F(\mu) + (1-t) F(\nu) +  t(1-t) \mathcal C (\nu,\mu)\,.
\end{equation}
Examples will be discussed in Section~\ref{sec:examples}. By definition, the case where $F$ is convex along flat interpolations, as considered in \cite{SongboLSI,Chen_etal}, corresponds to~\eqref{eq:courbure-cost} with $\mathcal C=0$.  

More specifically, we will require the following.

\begin{assu}
    \label{assu:semiconvexbound}
    There exists $\lambda\in[0,1)$ and $(\alpha_N)_{N\in\N} \in \R_+^{\N}$ such that \eqref{eq:courbure-cost} holds for some $\mathcal C$ satisfying 
\begin{equation}
\label{eq:cost-bound}
\int_{\R^{dN}} \mathcal C(\pi_{\bfx},\mub) \mu(\dd \bfx) \leqslant \frac{\lambda}N \mathcal H \po \mu  | \mub^{\otimes N}\pf + \frac{\alpha_N}N\,.
\end{equation}
 for all $N \in\N$ and  $\mu \in\mathcal P_2(\X^N)$.
\end{assu}

Of course, when $F$ is flat-convex, Assumption~\ref{assu:semiconvexbound} is trivially satisfied, with $\lambda=0=\alpha_N$. 

 For $N\in\N$ and $\mu \in\mathcal P_2(\X^{N})$, introduce the so-called $N$-particle free energy
 \[\mathcal F^N(\mu) = N\int_{\X^{N}} F(\pi_{\bfx}) \mu(\dd \bfx) + H(\mu)\,,\]
 which is such that  
 \[\mathcal F^N(\mu) - \mathcal F^N(\mu_\infty^N)  = \mathcal H\po \mu | \mu_\infty^N\pf\,.\]
 In particular,
 \[\alpha_N' := \mathcal F^N(\mu_\infty^N) - N \mathcal F(\mub) \leqslant  \mathcal F^N(\mub^{\otimes N}) - N \mathcal F(\mub) = N\int_{\X^{N}} \co F(\pi_{\bfx})  - F(\mub)\cf \mub(\dd \bfx) \]
 is $o(N)$ as $F$ is continuous with respect to $\mathcal W_2$ and $\mathcal W_2(\pi_{\bfX},\mub)$ with $\bfX\sim \mub^{\otimes N}$ goes to zero in probability thanks to \cite{FournierGuillin} as $N\rightarrow \infty$. In fact in many cases we can get that $\alpha_N'=\mathcal O(1)$, see Section~\ref{sec:examples}.

\begin{prop}
    \label{prop:entropycomparison}
Under Assumption~\ref{assu:semiconvexbound}, for all $\mu\in\mathcal P_2(\X^N)$
\[(1-\lambda) \mathcal H\po \mu |\mub ^{\otimes N}\pf    \leqslant   \alpha_N + \alpha_N' + \mathcal H(\mu |\mu _\infty^N)\,. \]

\end{prop}
\begin{proof}
    Dividing by $t$ and letting $t\rightarrow 0$ in~\eqref{eq:courbure-cost} gives for all $\nu, \eta\in \mathcal{P}_2(\X^N)$ the pent inequality
\begin{equation}
\label{eq:pent1}
\int_{\R^d} \frac{\delta F}{\delta m}(\eta)(\nu-\eta) \leqslant F(\nu)-F(\eta) + \mathcal C(\nu,\eta)\,.
\end{equation}
  Taking the expectation with $\nu = \pi_{\bfX}$ for $\bfX\sim \mu$ and $\eta=\mub$ and using \eqref{eq:cost-bound},
\begin{equation*}
     N \mathbb E \co F(\pi_{\bfX}) \cf  - N F(\mub) \geqslant N  \mathbb E \co \int_{\R^d} \frac{\delta F}{\delta m}(\mub,y)(\pi_{\bfX} - \mub)(\dd y)\cf  
- \lambda \mathcal H \po \mu | \mub^{\otimes N}\pf - \alpha_N\,.
\end{equation*}
Hence
\begin{eqnarray}
\mathcal F^N(\mu) - N \mathcal F(\mub)  & = & N \mathbb E \po  F(\pi_{\bfX}) - F(\mub) \pf + H(\mu) - N H(\mub) \nonumber  \\
& \geqslant & N  \mathbb E \co \int_{\R^d} \frac{\delta F}{\delta m}(\mub,y)(\pi_{\bfX} - \mub)(\dd y)\cf \nonumber \\
& & \ +\  H(\mu) - N H(\mub)  - \lambda \mathcal H \po \mu | \mub^{\otimes N}\pf - \alpha_N \nonumber  \\
  & = & (1-\lambda) \mathcal H\po \mu |\mub^{\otimes N}\pf    - \alpha_N \label{eq:FNetH} \,.
\end{eqnarray}
As a conclusion,
\[(1-\lambda) \mathcal H\po \mu |\mub^{\otimes N}\pf    \leqslant  \alpha_N   + \mathcal F^N(\mu) - N \mathcal F(\mub)  = \alpha_N + \alpha_N' + \mathcal H(\mu|\mub^N) \,.\]
\end{proof}

Let us now recall the result of \cite{MonmarcheLSI} (which follows and generalizes \cite{GuillinWuZhang,SongboLSI}). A probability measure $\nu$ on $\X$ is said to satisfy a Poincaré inequality with constant $\rho>0$ if for all $f\in\mathcal H^1(\nu)$,
\[\int_{\X} f^2 \dd \nu - \po \int_{\X} f\dd \nu\pf^2 \leqslant \frac1\rho \int_{\X}|\na f|^2 \dd \nu\,. \]

\begin{assu}\label{assu:Poincare}\

\begin{enumerate}[label=(\roman*)]
\item The partition function $Z_N = \int_{\R^{dN}} e^{-U_N}$ is finite, the free energy $\mathcal F$ is lower bounded and admits a minimizer $\mub$.\label{assu1i}
\item  There exists $\lambda'\geqslant 0$ such that~\eqref{eq:courbure-cost} holds with $\mathcal C(\mu,\nu)\leqslant \frac{\lambda'}2 \mathcal W_2^2(\mu,\nu) $.
\item The flat and intrinsic derivatives of $F$ order $1$ and $2$ exist and are continuous.
 \item There exists $M_{mm}^F \geqslant 0$ such that for all $m\in \mathcal P_2(\R^d),x,x' \in\R^d$, 
 $|D^2 F(m,x,x')|\leqslant M_{mm}^F$ (where $|\cdot|$ stands for the operator norm with respect to the Euclidean norm).  \label{assu1iv}
 \item For $N\in\N^*$, there exists $\rho_N>0$ such that for all $(x_2,\dots,x_N) \in\R^{d(N-1)}$, the probability density   proportional to $x_1\mapsto e^{-U_N(\bfx)}$ satisfies a Poincaré inequality with constant $\rho_N$.  \label{assu1v}
\item There exists $\bar{\rho}>0$ such that for all $\mu\in\mathcal P_2(\X)$, the probability measure   with density proportional to $\exp(-\frac{\delta F}{\delta m}(\mu,\cdot))$ satisfies a LSI with constant $\bar{\rho}$. \label{assu2ii}
\end{enumerate}
\end{assu}

\begin{thm}\label{thm:defectiveLSI}
Let $\varepsilon\in(0,1)$. Suppose Assumptions~\ref{assu:semiconvexbound} and \ref{assu:Poincare} hold with $\lambda<1/2$. Introduce the following notations:
\begin{align}
\tilde\lambda &=   \frac{2M_{mm}^F}{\bar{\rho} } \po 4 +  \frac{3M_{mm}^F }{2\bar{\rho}\varepsilon} \pf  \label{eq:lambdatilde}\\
\delta_N &= 4\bar{\rho}(1-\varepsilon)\po 2 \alpha_N  + \frac{M_{mm}^F d}{\bar{\rho}} \po \frac52 +  \frac{3M_{mm}^F }{4\bar{\rho}\varepsilon} \pf \pf \label{def:deltaN}\,.
\end{align}
 Then, for all $N>\tilde \lambda/(1-2\lambda)$,   $\mu_\infty^N$ satisfies a defective LSI with constants $(\rho_{N,*}',\delta_N)$ where
\[\rho_{N,*}' =   2 (1-\varepsilon) (1-2\lambda - \tilde\lambda/N)  \bar{\rho}\,.  \]
If moreover $ \rho_N - \lambda' - \frac{M_{mm}^F}{N}>0$,  then $\mu_\infty^N$ satisfies a tight LSI   with constant
\[\rho_{N,*} =   \rho_{N,*}'\po 1+\frac{\delta_N}{4\po \rho_N - \lambda' - \frac{M_{mm}^F}{N}\pf }\pf^{-1}   \,.  \]
\end{thm}

In particular, when $\liminf \rho_N > \lambda'$ and $\alpha_N = \mathcal O(1)$, then $\mu_\infty^N$ satisfies a tight LSI with constant independent from $N$ (i.e. Assumption~\ref{assu:LSI} holds for some $\rho$ with $\delta_N=0$).  More generally, according to Theorem~\ref{thm:defectiveLSI}, Assumption~\ref{assu:LSI} holds under Assumptions~\ref{assu:semiconvexbound} and \ref{assu:Poincare} provided $\lambda<1/2$ and $\alpha_N = o(N)$ in Assumption~\ref{assu:semiconvexbound}.

Notice that Assumption~\ref{assu:Poincare}$(vi)$ implies for $\mub$ a LSI with constant $\rho$, which implies a Talagrand inequality which, combined with Assumption~\ref{assu:Poincare}$(ii)$ and \cite{FournierGuillin}, shows that~\eqref{eq:cost-bound} holds with $\lambda =2 \lambda''/\rho$ (see next proof) and $\alpha_N$ of order $N^{1-2/d}$ (for $d>2$). This means that $\delta_N$ given by~\eqref{def:deltaN} is not bounded independently from $N$, but still negligible with respect to $N$, as required in Assumption~\ref{assu:LSI} and sufficient for~\eqref{eq:quadrisk2} to be useful (although it gives a convergence rate $N^{1/2-1/d}$ and thus doesn't recover the Monte Carlo rate). 

\begin{proof}
    This is essentially Theorem~2 and Corollary~3 of \cite{MonmarcheLSI}, with a single difference: in~\cite{MonmarcheLSI}, instead of~\eqref{eq:cost-bound} with $\lambda\in[0,1/2)$, what is assumed is that
    \begin{equation*}
\int_{\R^{dN}} \mathcal C(\pi_{\bfx},\mub) \mu(\dd \bfx) \leqslant \frac{\lambda''}N \mathcal W_2^2 \po \mu  , \mub^{\otimes N}\pf + \frac{\alpha_N}N\,,
\end{equation*}
for some $\lambda'' <\bar{\rho}/4$. But in fact this is only used in conjunction with the Talagrand inequality satisfies by $\mub$ (implied by the LSI given by Assumption~\ref{assu:Poincare}$(v)$) which reads
\[\mathcal W_2^2 \po \mu  , \mub^{\otimes N}\pf  \leqslant \frac{2}{\bar{\rho}} \mathcal H \po \mu  | \mub^{\otimes N}\pf \,. \]
In other words, the assumptions in \cite{MonmarcheLSI} imply Assumption~\ref{assu:semiconvexbound} with $\lambda = 2\lambda''/\bar{\rho}$, and as can be checked by following the proof of \cite[Theorem 1]{MonmarcheLSI} this is in fact all we need to 
establish Theorem~\ref{thm:defectiveLSI}. More precisely, with our present assumptions \cite[Equation (23)]{MonmarcheLSI} is replaced by
\[
\mathcal I \po \mu |\mu_{\infty}^N\pf \geqslant  4\bar{\rho} (1-\varepsilon) \co  \mathcal F^N(\mu) - N \mathcal F(\mub) - (\lambda + \tilde \lambda / N) \mathcal H \po \mu |\mub ^{\otimes N}\pf - \tilde \alpha_N \cf \,, \]
with $\tilde \lambda $ given by \eqref{eq:lambdatilde}  and
\[\tilde \alpha_N  =  \alpha_N  + \frac{M_{mm}^F d}{\bar{\rho} } \po \frac52 +  \frac{3M_{mm}^F }{4\rho\varepsilon} \pf \,.  \]
Applying~\eqref{eq:FNetH} concludes the proof of the defective LSI. The tight LSI under the additional condition that $ \rho_N - \lambda' - \frac{M_{mm}^F}{N}>0$ is exactly \cite[Corollary 3]{MonmarcheLSI}.
\end{proof}

\subsection{Conditions for uniform moments}\label{sec:conditionLyap}

When the position is in the compact torus, Assumption~\ref{assu:moments} only requires a time-uniform sixth moment for the velocity. We can establish the following, proven in Section~\ref{sec:proof-Lyap-torus}.

\begin{prop}
    \label{prop:LyapTorus}
    Assume that $\X=\T^d$ and 
    \begin{equation}
        \label{eq:assuLyapTorus}
    \|D F\|_\infty = \sup_{\mu\in\mathcal P_2(\T^d)} \sup_{x\in\T^d} |DF(\mu,x)| < \infty\,, \end{equation}
        and write
    \begin{equation}
        \label{def:LyapTorus}
    \bfV(\bfz)  = \sum_{i=1}^N |v_i|^6\,.
    \end{equation}
    Then, for all $\bfz\in\X^N\times\R^{dN}$,
    \begin{equation}
        \label{eq:LyaptorConclusion}
         \mathcal P \bfV(\bfz) \leqslant   (1-\gamma h)  \bfV(\bfz) + Nh \co 766   \gamma  d^3    + \frac{\|D F\|_\infty^6  }{\gamma^5}\cf \,.
    \end{equation}
\end{prop}
Since $\nu_n^N = \nu_0^N \mathcal P^n$, by induction on $n$, an immediate consequence of~\eqref{eq:LyaptorConclusion} is that, provided $\nu_0^N(\bfV)<\infty$ then, for all $n\in\N$,
\begin{align*}
  \nu_n^N(\bfV) &\leqslant (1-\gamma h)^n  \nu_0^N(\bfV) + \co 1 - (1-\gamma h)^n  \cf  N \co 766     d^3    + \frac{\|D F\|_\infty^6  }{\gamma^6}\cf\\
  &\leqslant \max\po  \nu_0^N(\bfV) ,   N \co 766     d^3    + \frac{\|D F\|_\infty^6  }{\gamma^6}\cf\pf \,,
\end{align*}
Dividing by $N$ and assuming that particles are indistinguishable under $\nu_0^N$, this implies Assumption~\ref{assu:moments}. If the initial velocities are independent standard Gaussian variables then $\nu_0^N(\bfV) \leqslant 15 N d^3$ (see \eqref{eq:G6}). So, in terms of $d$, as long as  $\|DF\|_\infty$ is at most of order $\sqrt{d}$ (which is the order of the diameter of $\T^d$), Proposition~\ref{prop:LyapTorus} shows that $C_1$ in~\eqref{eq:C1} is independent from $d$, provided the bounds on the derivatives in Assumption~\ref{assu:regularity} and thus the constants $L_ i$ in Proposition~\ref{lem:regularity} are independent from $d$.

\bigskip

The computations are more involved in the case $\X=\R^d$. They are inspired by the continuous-time case,  analyzed in \cite{talay2002stochastic,MATTINGLY2002185}. In fact, for the reader's convenience, we briefly recall the continuous-time computations in Section~\ref{sec:LyapContinuous}, first without mean-field interaction and then with it. This gives the high-level structure of the computations in the discrete-time case, which are conducted in Section~\ref{sec:LyapDiscrete}, leading to the proof of Proposition~\ref{prop:LyapunovRd} stated below, established under the following set of conditions:

\begin{assu}
    \label{assu:LyapunovRd}
    The energy can be decomposed as $F(\mu) = \int V\dd \mu + F_1(\mu)$ for some $V\in\mathcal C^2(\R^d)$ and $F_1:\mathcal P_2(\R^d)\rightarrow \R$. Moreover, $\na V(0)=0$ and  there exist $M,\lambda,r,K,L,c_0,c_1,R_0,R_1>0$ such that for all $\mu\in\mathcal P_2(\R^d),y\in\R^d$,
\begin{equation}
\label{condition:lambdaM}
|DF_1(\mu,y)| \leqslant M\sqrt{d} + \lambda |y| + \lambda \int_{\R^d} |x| \mu(\dd x)\,,
\end{equation}
and
\begin{equation}
    \label{condition:VLyap}
|\na^2 V(y)|\leqslant L\,,\qquad -y\cdot \na V(y) \leqslant  - r |y|^2 + Kd \,,\qquad dR_1+ c_1 |y|^2 \geqslant  V(y) \geqslant dR_0 + c_0 |y|^2 \,. \end{equation}
\end{assu}

These conditions are more general than~\cite[Assumption H4]{Camrudetal}. First, $F_1$ is not necessarily of the form $F_1(\mu)=\int_{\R^{2d}} W \mu^{\otimes 2}$, and one of conditions in \cite{Camrudetal}, written in terms of the Lipschitz norm of $\na W$, is essentially similar to assuming~\eqref{condition:lambdaM} with $M=0$ (which is important as, in Proposition~\ref{prop:LyapunovRd}, there is no restriction on $M$. In other words, in \cite{Camrudetal}, the non-linear force is a small Lipschitz drift, while in our case we allow for arbitrarily large bounded drift, plus a small linearly-growing part - cf. restriction on $\lambda$ in Proposition~\ref{prop:LyapunovRd}).

The conditions that the lower-bound on $V$ in \eqref{condition:VLyap} is positive and that $\na V(0)=0$, which together with~\eqref{condition:VLyap} imply that
\begin{equation}
    \label{cond:VlyapGradient}
    |\na V(y)|\leqslant L |y|,
\end{equation}
 can always be enforced without loss of generality by a suitable translation of $\bfx$ and of $V$.

For some $\alpha>0$ to be chosen later (in \eqref{eq:choicealpha}), set
\begin{equation}
    \label{def:LyapRd}
    H_0(y,w)  =  V(y) + \frac12 |w|^2\,,\qquad \varphi(z_i)=   H_0(z_i) + \alpha x_i\cdot v_i \,,\qquad\bfV(\bfz) = \sum_{i=1}^N\varphi^3(z_i)  \,.
\end{equation}
for $y,w\in \R^d$ and $\bfz=(z_i)_{i=1}^N\in \R^{2dN}$ with $z_i=(x_i,v_i)$.
Assuming~\eqref{condition:VLyap} and that $0<\alpha \leqslant \sqrt{c_0/2} $ gives
\begin{equation}
\label{eq:boundphi}
dR_0+ \frac{c_0}2 |x_i|^2 + \frac{1}{4}|v_i|^2  \leqslant \varphi(z_i) \leqslant dR_1+  \frac{3c_1}{2}  |x_i|^2 + \frac{3}{4}|v_i|^2 \,,
\end{equation}
where we used that necessarily $c_0 \leqslant c_1$ in the right hand side.

We will track the dependency in $N,h,d$, but not on $\mathfrak{p}:=(M,L,c_0,R_0,R_1,r,K,\gamma)$, except for the restriction on $\lambda$ (in \eqref{eq:lambda0}). For this reason, in the next result and throughout its proof, for the sake of conciseness, we will write that $C\in\mathcal C(\mathfrak{p})$ when $C$ is a generic constant that depends only of $\mathfrak{p}$. There will also be a restriction on $h$ in terms of $\mathfrak{p}$ (namely, $h\leqslant h_0$ for some $h_0\in\mathcal C(\mathfrak{p})$; recall besides that with~\eqref{def:eta} we assumed from the start that $h<\gamma^{-1}$), however, we will not make it explicit, because we think of $\mathfrak{p}$ as parameters independent from the ambient dimension $d$ (the assumed dependency in $d$ being already explicit in~\eqref{condition:lambdaM} and \eqref{condition:VLyap}), and in view of the $d^3$ scaling in the numerical error in Theorem~\ref{thm:GlobalEntropyDecay},   $h$ has to be taken of order $d^{-3/4}$, making the constraint in terms of $\mathfrak{p}$ always satisfied for large ambient dimension $d$ (besides, making explicit the constraint on $h$ is not very informative and makes the computations much more tedious, hiding the important steps of the proofs in a technical flood).

\begin{prop}\label{prop:LyapunovRd}
Under Assumption~\ref{assu:LyapunovRd}, set
\begin{align}
  \alpha & = \min \po \frac{\gamma}{2} \po \frac{ 2\gamma^2 }{r} + \frac{19}{12}\pf^{-1}, \sqrt{\frac{c_0}2} \pf  \label{eq:choicealpha} \\
  \theta & = \frac12 \min\po \frac{\alpha r }{5c_1}, \gamma\pf  \\
  \lambda_0 &=  \min\po \frac{r}{3}, \frac{2\alpha}{3}, \frac{ r\alpha c_0^2}{176(1+\alpha)^3} ,  \frac{2\theta}{1+\alpha}\po \frac{16}{c_0^3 N}  + 2\pf^{-1}   \pf\,. \label{eq:lambda0}
\end{align}
There exist $h_0 ,C \in\mathcal C(\mathfrak{p})$ such that, assuming that $h \leqslant h_0$ and $\lambda \leqslant \lambda_0$ and considering $\bfV$ as in~\eqref{def:LyapRd} (with $\alpha$ given by~\eqref{eq:choicealpha}) then, for all $N\in\N$ and $\bfz\in\R^{2dN}$,
 \[
 \mathcal P(\bfV)(\bfz) 
  \leqslant  \po 1- \theta h  \pf  \bfV(\bfz) + 
    CN hd^3   \,.
\]
\end{prop} 

With respect to \cite[Theorem 3]{Camrudetal}, the condition on $\lambda$ is explicit (on the other hand, in \cite{Camrudetal}, a more general family of generalized HMC schemes are considered, which makes the analysis more intricate).

As in the case of the torus, when particles are indistinguishable, together with~\eqref{eq:boundphi}, this implies Assumption~\ref{assu:moments}. Moreover, under Assumptions~\ref{assu:regularity} and \ref{assu:LyapunovRd}, thanks to~\eqref{eq:lipbound1},  $C_1$ in~\eqref{eq:C1} depends only on $\mathfrak{p}$, $\nu_0(\bfV)/d^3$ and the parameters $L_i$ of Lemma~\ref{lem:regularity}, but doesn't have any additional dependency to $d$.

\subsection{Related works} \label{sec:relatedworks}

We have already discussed in the introduction the existing works most closely related to our study, \cite{BouRabeeSchuh,Camrudetal,Chen_etal,GuillinMonmarche,fu2023mean}, concerned with non-asymptotic convergence bounds for mean-field models with kinetic processes, and commented in details their relation with the present study. Let us mention other related references in a broader scope.

\paragraph{Non-asymptotic bounds for MCMC.} There is a plethoric classical literature on convergence rates for Markov chain. Over the last decade, motivated by modern high-dimensional problems, much focus has been put on non-asymptotic quantitative bounds (with, in particular, explicit dependency in the dimension, taking into account time discretization errors when needed). Because this has been a very active topic,  there is still a huge literature now even if we restrain to recent works in this direction,  and thus we won't try to be exhaustive here but rather to highlight different methods and directions. Explicit coupling methods have first been used for log-concave targets \cite{dalalyan2017theoretical,MonmarcheSplitting,DurmusULA,leimkuhler2023contraction} and then  in the non-convex case with more sophisticated couplings \cite{BouRabeeEberleZimmer,schuh2024convergence,chak2023reflection,cheng2018underdamped}, often with unadjusted schemes (because coupling methods don't require explicit expressions for invariant measures and thus they are robust when applied to discretization schemes) but they also appear in the analysis of Metropolis-adjusted schemes \cite{BouRabeeOberdorster}. For Metropolis-adjusted methods or more generally for reversible Markov chains, conductance methods have recently proven efficient \cite{andrieu2024explicit,wu2022minimax,chen2020fast}. Before the present work, entropy methods have been used in \cite{Chatterji,VempalaWibisono} on first-order Euler schemes of (overdamped and underdamped) Langevin diffusion and in \cite{Camrudetal}  (upon which we build) for underdamped Langevin second-order splitting schemes.

\paragraph{Numerical  analysis for mean-field SDE.}  One of the classical approach to numerical analysis of SDE is backward error analysis (concerned with weak error of estimators). Its application to mean-field analysis traces back at least to  \cite{bossy1997stochastic,malrieu2003convergence}. Strong error analysis is also possible~\cite{ding2021euler,chen2023wellposedness,bao2019approximations}, usually not uniform in time (except under strong contractivity assumptions) but in some cases they can be combined with Wasserstein contractions to get uniform in time Wasserstein bias estimates as in e.g. \cite{Gouraud_etal,durmus2024asymptotic}. In terms of numerical analysis, a topic which is specific to mean-field interacting system is the use of particle mini-batch, which is a form of stochastic gradient method: in order to reduce the numerical complexity, instead of computing the $N^2$ interaction in the system (for pairwise interactions), only a random subsample of those are considered at each step-size. The convergence of this method for small step-size is studied in \cite{Jinetal,ye2024error,li2020random,jin2023ergodicity} and its effect on transition phases in \cite{GuillinRandomBatch}.  Finally, similarly to long-time convergence, there has recently been active research to get non-asymptotic convergence bounds (in terms of $N$ and of the number of iterations) with an optimization viewpoint, in the spirit of e.g. \cite{Suzukietal,fu2023mean} which have already been discussed in the introduction, see also \cite{LascuetalJKOMeanField} and references within these works.

\paragraph{Convergence for non-linear continuous-time kinetic Langevin diffusion.} The analysis of practical MCMC samplers  based on discretizations of the kinetic Langevin samplers builds upon works concerned with the long-time convergence of the continuous-time process, which is already non-trivial because of its degenerate noise structure. For the standard linear kinetic Langevin diffusion, we refer to the bibliographical discussions in \cite[Section 3.3]{Camrudetal} and \cite[Section 1.1]{GouraudJournelM}, and in the following we focus on the non-linear process, corresponding to the Vlasov-Fokker-Planck equation. Lyapunov/minorization methods working with the total variation distance (or more general $V$-norms) are usually not very well suited for mean-field particle systems and non-linear processes, in particular diffusion processes, because of the scaling behaviour of these distances with respect to dimension, hence with $N$ (that being said, see \cite{monmarche2023elementary,journel2022uniform} for recent examples in this direction, with particle-wise total variation). Similarly, $L^2$ approaches are not the most convenient for non-linear processes and for now have been restricted to perturbative regimes where the non-linearity is small \cite{cesbron2024vlasov,herau2007short,addala20212}. Contrary to total variation distance and $L^2$ norm, Wasserstein distance scales well with dimension, so that particle systems can be handled with couplings. This has first been done for the kinetic Langevin diffusion in \cite{bolley2010trend} with parallel coupling for convex potentials and then in \cite{schuh2024global,GMLB,kazeykina2024ergodicity} with reflection coupling in non-convex cases, in all cases with a smallness condition on the Lipschitz constant of the non-linearity. Finally, an approach that was quite successful for mean-field diffusion processes is entropy/free energy methods (with, more specifically in the kinetic Langevin case, the modified entropy of Villani \cite{villani2009hypocoercivity}). This is due notably to the nice scaling properties of the entropy (and Fisher information) with respect to the dimension (the relative entropy with respect to $\mu_\infty^N$, divided by $N$, converging to the free energy $\mathcal F$ up to an additive constant), and also to the tensorization property of functional inequalities and their capacity to take into account some structure of the non-linearity (for instance flat-convexity)  \cite{SongboLSI,MonmarcheLSI,GuillinWuZhang}, by contrast to coupling methods. The first applications of this approach to the non-linear Langevin diffusion were obtained by working at the level of the particle system \cite{GuillinMonmarche,MONMARCHE20171721} (which require uniform in $N$ LSI for $\mu_\infty^N$). More recently, the computations have been conducted directly at the level of the non-linear limit equation \cite{Chen_etal} (although working directly with the free energy had been done previously in e.g. \cite{kazeykina2024ergodicity,DuongTugaut18} but for qualitative convergence without rates), as in the overdamped case \cite{malrieu2001logarithmic}.   The interest (with respect to \cite{GuillinMonmarche,MONMARCHE20171721}) was that it is was sufficient to establish  some non-linear log-Sobolev inequalities, which are implied by uniform-in-$N$ LSI for $\mu_\infty^N$ but can be simpler to establish (although eventually, very recently, S. Wang proved in \cite{SongboLSI} that, in the flat convex and under essentially the same conditions considered in previous works in the non-linear equation, the uniform-in-$N$ LSI does in fact hold), in particular in the flat-convex case which led to a number of works in the machine learning community, first for overdamped and then underdamped Langevin diffusion, see \cite{Suzukietal,fu2023mean,Chizat,nitanda2022convex} and references within.


\section{Examples}\label{sec:examples}

\subsection{Bounded convex non-linearity}

Here we consider  settings similar to \cite{Chizat,chizat2018global,Chen_etal,nitanda2022convex}, motivated by mean-field models in machine learning algorithms, such as shallow neural networks \cite{Szpruch,mei2018mean,mei2019mean}, kernel Maximum Mean Discrepancy minimization \cite{Chizat}, low-rank tensor decomposition \cite{haeffele2017global} and so on, see e.g. \cite{chizat2018global} and references within for other examples. In these applications, contrary to classical cases from statistical physics, the non-linear part of $F$ is not given by some pairwise interaction potential as in~\eqref{eq:Fmu_potentials}. However, $ F$ is convex along flat interpolations, meaning that~\eqref{eq:courbure-cost} holds with $\mathcal C=0$. In particular, Assumption~\ref{assu:semiconvexbound} is trivially satisfied with $\lambda=\alpha_N=0$.

Let us discuss briefly simple additional conditions that can typically hold in these applications. We won't discuss the estimates in detail, as this should be done separately for each model in each particular setting of interest to get informative sharp rates (see next section for details in a specific example).

More precisely, let us assume that $F(\mu) = \int_{\R^d} V \dd \mu + F_1(\mu)$ for some flat-convex $F_1$. Typically, in  many algorithms, $V$ is a penalization or regularization chosen by the user, often taken as $V(x)=\frac{r}{2}|x|^2$. More generally, let us assume that it is $\mathcal C^4$ with derivatives of order $2$, $3$ and $4$ bounded, and that it is the sum of a strongly convex and a bounded function. We assume furthermore that $D F_1$ and $D^2 F_1$ are bounded, and that the other regularity conditions of Assumption~\ref{assu:regularity} are satisfied by $F_1$

 As an illustration, in the case of shallow neural networks $F_1(\mu)$ is given by $\mathbb{E}_{(X,Y)}[l(Y,\phi_{\mu}(X))]$ where $l$ is a smooth and convex loss function like the quadratic loss or the logistic loss and $(X,Y)$ are input and output data, respectively, which are distributed according to an empirical or true data distribution. The function $\phi_{\mu}$ satisfies $\phi_{\mu}(x)=\int \phi_{\theta}(x) \mu(\dd \theta)$ where $\phi_{\theta}=\sigma( \theta \cdot x)$ corresponds to a single neuron with trainable parameter $\theta \in \X$ and an activation function $\sigma$. Our conditions above hold if $\sigma$ is smooth and bounded (e.g. the sigmoid function).
 
 Under the previous conditions,  Assumption~\ref{assu:LyapunovRd} holds, with~\eqref{condition:lambdaM} being satisfied for some $M>0$ and $\lambda=0$. As a consequence, Proposition~\ref{prop:LyapunovRd} applies without any restriction on the model (only a restriction on the step-size $h$), so that Assumption~\ref{assu:moments} holds (for $h\leqslant h_0$ for some $h_0 \in\mathcal C(\mathfrak{p})$, in particular independent from $N$). Assuming that $\frac{\delta F_1}{\delta m}
$ is also bounded, all conditions of Assumption~\ref{assu:LSI}  can be established using classical perturbation arguments for LSI, as e.g. in \cite[Section 3]{Chen1} (see also next section), with $\rho_N$ which is in fact bounded away from $0$ uniformly in $N$. As a consequence, Theorem~\ref{thm:defectiveLSI} applies and, since $\lambda'=0$, we get a tight LSI (since we are in a flat-convex setting, in fact, the result of \cite{SongboLSI} applies), namely Assumption~\ref{assu:LSI} is satisfied with $\delta_N=0$.

As a conclusion, in this simple (yet relatively general) setting, we have checked all conditions for Theorem~\ref{thm:GlobalEntropyDecay} (and also~\eqref{eq:compareEntropy} thanks to Proposition~\ref{prop:entropycomparison}), leading to the quantitative error bounds~\eqref{eq:quadrisk1} and \eqref{eq:quadrisk2}.

\subsection{A non-flat-convex example}\label{sec:exnonconvex}

This example is taken from~\cite{MonmarcheLSI}. Consider
\[F(\mu) = \int_{\R^d} V \dd \mu + \frac12\int_{\R^{2d}} W(x-y) \mu(\dd x)\mu(\dd y)\,,\]
where  $W=W_1+W_2$, 
\[W_1(x-y) = L e^{-|x-y|^2}, \qquad W_2(x-y):= s |x-y|^2\]
for some $L,s>0$. In the long-range (i.e. when $|x-y|$ is large enough), the interaction is attractive while, assuming that $L>s$, the repulsive potential $W_1$ prevails at short range. This short-range repulsion/long-range attraction is met in many interacting systems. Let us check the conditions of Theorem~\ref{thm:GlobalEntropyDecay} under suitable conditions.

\medskip

\emph{Regularity conditions:} Since 
\[DF(\mu,y) = \na V(y) +  \mu \star \na W(y)\,,\quad D^2 F(\mu,x,y) = \na^2 W(x-y)\,,\quad D^3 F(\mu,x,y,z)=0\,,\] 
Assumption~\ref{assu:regularity} is met as soon as $\na^2 V$, $\na^3 V$ and $\na^4 V$ are bounded. 

\medskip

\emph{Log-Sobolev inequality:} As shown in \cite{MonmarcheLSI}, \eqref{eq:courbure-cost} holds with
\begin{equation}
    \label{eq:Cexemple}
    \mathcal C(\mu,\nu) = \alpha \po \int_{\R^d} x \mu(\dd x) -\int_{\R^d} x \nu(\dd x) \pf^2\,.
\end{equation}
The rest of Assumption~\ref{assu:LSI} are checked in \cite{MonmarcheLSI} under the following conditions: assume that the confining potential is decomposed as $V=V_1+V_2$ where $V_1$ is bounded and, for simplicity later on, $V_2(x)=\frac{r}{2}|x|^2$ for some $r>0$. Then, the uniform LSI for $\mu_\infty^N$ (i.e. \eqref{eq:defectiveLSI} with $\delta_N=0$) is proven in \cite[Section 1.4.4]{MonmarcheLSI} assuming furthermore
\[4 s < r  e^{-\|V_1\|_\infty - L}\,.\]
In other words, this gives Assumption~\ref{assu:LSI}.

\medskip

\emph{Uniform moments:} Writing $F_1(\mu) = \frac12 \int \mu\star W \dd \mu $, we get 
\[|DF_1(\mu,y)| = |\mu\star \na W (y)| \leqslant \|\na W_1\|_\infty + 2s  \po |y| + \int_{\R^d}|x|\mu(\dd x)\pf\,, \]
 hence  \eqref{condition:lambdaM}  with $\lambda=2s $. Assuming that $\na V_1$ and $\na^2 V_1$ are  bounded, we get~\eqref{condition:VLyap} with $c_0=c_1=r/2$ (and the same $r$).  To simplify the expressions in Proposition~\ref{prop:LyapunovRd}, take for instance the friction parameter as $\gamma=\sqrt{r}$. This gives, with the notation of Proposition~\ref{prop:LyapunovRd},
 $\alpha = \sqrt{r}/7$ and $\theta = \sqrt{r}/35$. Then, the three terms involved in $\lambda_0$ in \eqref{eq:lambda0} are not homogeneous in terms of $r$, and thus we can clarify the condition by rescaling the process  as in e.g. \cite[Section 1.3]{Monmarcheidealized}, here with the change of variable $x \mapsto \sqrt{r} x$ (together with a rescaling of the step-size for consistency with the velocities having standard Gaussian target, see \cite[Section 1.3]{Monmarcheidealized}). This means that, in the previous conditions,  $s$ is replaced by $s/r$ and $r$ is replaced by $1$, so that in order to apply Proposition~\ref{prop:LyapunovRd} to the rescaled process we have to assume that 
 \[\frac{2 s}r \leqslant  \lambda_0 = \min\po \frac{1}{3}, \frac{2}{21}, \frac{  (1/2)^2}{7\times 176(1+1/7)^3} ,  \frac{2/35}{1+1/7}\po \frac{16}{(1/2)^3 N}  + 2\pf^{-1}   \pf\,, \]
which is implied for any $N\geqslant 1$ by
\[ s  \leqslant \frac{r}{15713}\,.\]
Hence, under this (probably non-sharp) condition, Proposition~\ref{prop:LyapunovRd} holds provided the step-size is small enough, which implies Assumption~\ref{assu:moments} (for the rescaled process, and thus for the process itself).

\medskip

As a conclusion, under the conditions discussed above, Theorem~\ref{thm:GlobalEntropyDecay} applies (notice that there is no restriction on $L$; in the sense that for any fixed $L>0$ and $V$, all conditions are met provided $s$ is small enough; this is in contrast to results involving some smallness condition on the Lipschitz constant of the interaction, as in e.g. \cite{Camrudetal,BouRabeeSchuh}). In order then to apply Proposition~\ref{prop:quadrisk}, we can notice that Assumption~\ref{assu:semiconvexbound} (with $\mathcal C$ given in~\eqref{eq:Cexemple}) follows from the study in \cite{MonmarcheLSI} (with $\alpha_N = \mathcal O(1)$), which thanks to Proposition~\ref{prop:entropycomparison} gives~\eqref{eq:compareEntropy} (and it is then not difficult, similarly to the proof of \cite[Lemma 17]{GuillinWuZhang}, to see that, in the present case, $\alpha_N'=\mathcal O(1)$, so that $\eta_N  = \mathcal O(1)$ in \eqref{eq:compareEntropy}). That being said, as discussed after Proposition~\ref{prop:quadrisk}, it may be more interesting to use \eqref{eq:quadrisk1} instead of \eqref{eq:quadrisk2} when sharp propagation of chaos at stationarity is known. Such a result follows from \cite[Proposition 3.8]{MonmarcheRenWang}, for $s$ small enough, implying that, indeed,
\[ \TV (\mu_\infty^{2,N},\mub^{\otimes 2}) = \mathcal O\po \frac1N\pf\,. \]

Notice that some restriction on $s$ being small enough is to be expected to ensure uniqueness of the solution $\mu_*$ of \eqref{eq:mu*self-consistent} and then get global convergence bounds as we obtain. See \cite{MonmarcheReygner} and references within for examples with non-uniqueness.

\section{Proofs of the main results}\label{sec:mainproof}

As a first step, let us recall the approximate modified entropy dissipation established in \cite{Camrudetal}. For $a>0$, define the modified entropy
\begin{equation}
    \label{def:modifiedEntropy}
\mathcal L_a(\nu) = \mathcal H(\nu|\nu_\infty^N) + a\int_{\X^N\times\R^{dN}} \left|(\na_x + \na_v) \ln \frac{\nu}{\nu_\infty^N}\right|^2 \nu \,,
\end{equation}
for a probability density $\nu\in\mathcal P_2(\X^N\times\R^{dN})$.

The following is established in the proof of~\cite[Theorem~5]{Camrudetal} (just before the LSI is applied; strictly speaking this is only written in \cite{Camrudetal} for $\X=\R^d$ but the proof also straightforwardly applies when $\X=\T^d$).

\begin{prop}\label{prop:Camrud}
   Under Assumptions~\ref{assu:regularity} and \ref{assu:moments}, setting $a$ as in~\eqref{eq:a}, for all $n\in\N$ and $N\geqslant 1$,
    \begin{equation}
        \label{eq:approximateDissip}
    \mathcal L_a(\nu_{n+1}  ) - \mathcal L_a(\nu_n ) \leqslant - \frac{a h}{3}   \mathcal I(\nu_{n+1} |\nu_\infty^N ) + \po 9 + \frac1a\pf  C_1  N d^3    h ^5 \,,
    \end{equation}
    with $C_1$ given by~\eqref{eq:C1}.
   \end{prop}

The proof of Theorem~\ref{thm:GlobalEntropyDecay} then easily follow by using the defective LSI~\eqref{eq:defectiveLSI}.

\begin{proof}[Proof of Theorem~\ref{thm:GlobalEntropyDecay}]
    First, the standard Gaussian distribution $\R^{dN}$ satsifying an LSI with constant $1$, the tensorization property of LSI (see \cite[Proposition 5.2.7]{BakryGentilLedoux}) shows that $\nu_N$ also satisfies the defective LSI
        \begin{equation}
        \label{eq:defectiveLSI_nuN}
        \forall \nu \in \mathcal P_2(\X^N\times\R^{dN}),\qquad  \mathcal H(\nu|\nu_\infty^N) \leqslant  \max(1,\rho^{-1}) \mathcal I(\nu|\nu_\infty^N) + \rho^{-1} \delta_N\,.
    \end{equation}
    As a consequence,
    \[
        \mathcal L_a(\nu)  \leqslant \mathcal H(\nu|\nu_\infty^N) + 2 a \mathcal I(\nu|\nu_\infty^N)  \leqslant \po \max(1,\rho^{-1})  + 2a \pf \mathcal I(\nu|\nu_\infty^N) + \rho^{-1} \delta_N \,.
    \]
    Plugging this in~\eqref{eq:approximateDissip} gives
    \[ \mathcal L_a(\nu \mathcal P ) - \mathcal L_a(\nu ) \leqslant - \kappa h \mathcal L_a(\nu \mathcal P )  + \kappa h  \rho^{-1}\delta_N +  \po 9 + \frac1a\pf  C_1  N d^3     h ^5  \]
    with $\kappa$ given in~\eqref{eq:a},  and then, in view of the definition of $C_2$,
      \[ \mathcal L_a(\nu \mathcal P ) \leqslant \frac{1}{1+\kappa h} \mathcal L_a(\nu )  + \kappa h  \rho^{-1}\delta_N  + \kappa C_2  N d^3      h ^5 \,.   \]
  Conclusion then follows from
  \begin{align*}
      \mathcal H(\nu_n^N|\nu_\infty^N) &\leqslant \mathcal L_a(\nu_n^N) \\
      & \leqslant (1+\kappa h)^{-n} \mathcal L_a(\nu_0^N) + \kappa h  \po  \rho^{-1} \delta_N + C_2  N d^3 h^4\pf  \sum_{j\in\N}(1+\kappa h)^{-j} \\
      &\leqslant (1+\kappa h)^{-n} \po \mathcal H(\nu_0^N|\nu_\infty^N) + 2 a \mathcal I(\nu_0^N|\nu_\infty^N)\pf  +   \rho^{-1} \delta_N +   C_2  N d^3 h^4\,.  
  \end{align*}
\end{proof}

\begin{lem}
    \label{lem:entropyDecay}
    Assuming~\eqref{eq:compareEntropy},  denoting by $\mu_n^{k,N}$ the $k$-marginal of $\mu_n^N$ (i.e. the law of $(X_1,\dots,X_k)$ when $\mathbf{X}\sim \mu_n^N$), it holds
  \[
        \mathcal H\po \mu_n^{k,N}|\mub^{\otimes k}\pf     \leqslant \frac{k}{N} \po R \mathcal H(\nu_n^N |\nu_\infty^N) + \eta_N \pf  \,.
    \]
\end{lem}

\begin{proof}
    By the classical sub-additivity property of the relative entropy (see e.g. \cite[Lemma 5.1]{Chen1}) and the indistinguishability of particles, for $k\leqslant N$,
    \[\mathcal H\po \mu_n^{k,N}|\mub^{\otimes k}\pf \leqslant   \frac{k}{N}  \mathcal H\po \mu_n^N|\mub^{\otimes N}\pf \,.  \] 
    Thanks \eqref{eq:compareEntropy}, we conclude with
    \[
        \mathcal H\po \mu_n^{k,N}|\mub^{\otimes k}\pf \leqslant \frac{k}{N} \po R \mathcal H(\mu_n^N |\mu_\infty^N) + \eta_N \pf   \leqslant \frac{k}{N} \po R \mathcal H(\nu_n^N |\nu_\infty^N) + \eta_N \pf  \,.
    \]
\end{proof}

\begin{proof}[Proof of Proposition~\ref{prop:quadrisk}]
 To prove~\eqref{eq:quadrisk1}, writing $g(\bfX) = \frac1N \sum_{i=1}^N f(X_i) - \mub(f) $ and considering $\mathbf{Y} \sim \mu_\infty^N$
\[\mathbb E \co g^2(\bfX) \cf \leqslant  \mathbb E \co g^2(\mathbf{Y}) \cf + \|g\|_\infty^2  \TV(\mu_n^N,\mu_\infty^N)\]
The second term is bounded thanks to Pinsker's inequality and $\|g\|_\infty \leqslant 2 \| f \|_\infty$. For the first one, writing $\bar{f} = f - \mub(f)$,
\[\mathbb E \co g^2(\mathbf{Y})\cf = \frac1N \int_{\X} \bar {f}^2(x) \mu_\infty^{1,N}(\dd x) + \po 1 - \frac1N\pf  \int_{\X^2} \bar {f}(x) \bar {f}(y) \mu_\infty ^{2,N}(\dd x,\dd y) \,. \]
The first term is bounded by $\|\bar {f}\|_\infty^2 /N \leqslant 4\|f\|_\infty^2 /N$ and, using that $\bar{\mu}_*\bar {f}=0$,  
\[
\int_{\R^{2}} \bar {f}(x)\bar {f}(y) \mu_{\infty}^{2,N}(\dd x, \rmd y)=\int_{\R^{2d}} \bar {f}(x)\bar {f}(y) (\mu_{\infty}^{2,N}-\bar \mu_*^{\otimes 2})(\dd x, \rmd y) \le  \|\bar {f} \|_{\infty}^2\TV (\mu_\infty ^{2,N},\mub^{\otimes 2}) \,.
\]
Conclusion follows $\|\bar {f}\|_\infty \leqslant 2 \|f\|_\infty$.

Now, under~\eqref{eq:compareEntropy}, we bound directly $\mathbb E \co g^2(\mathbf{X})\cf$ as we did with $\mathbb E \co g^2(\mathbf{Y})\cf$, namely 
\[\mathbb E \co g^2(\mathbf{X})\cf \leqslant \frac{4\|f\|_\infty^2}{N} +  4\|f \|_{\infty}^2\TV (\mu_n ^{2,N},\mub^{\otimes 2}) \,. \]
This gives~\eqref{eq:quadrisk2} thanks to the Pinsker's inequality and Lemma~\ref{lem:entropyDecay} (applied with $k=2$).
\end{proof}

\section{Moment bounds}\label{sec:proofLyap}
\subsection{Case of the torus}\label{sec:proof-Lyap-torus}

In this section, devoted to the proof of Proposition~\ref{prop:LyapTorus}, $\X=\T^d$, \eqref{eq:assuLyapTorus} is enforced and  we consider the Lyapunov function $\bfV$ given by~\eqref{def:LyapTorus}. For an initial condition $\bfz$, denote by $\bar{\bfz}=(\bar{\bfx},\bar{\bfv})$ the state after one Verlet transition, i.e.,
\begin{equation}
    \label{eq:Verletxivi}
    \left\{\begin{array}{lcl}
\bar{x}_i &= & x_i + h v_i - \frac{h^2}{2} \na_{x_i} U_{N}(\bfx) \\
    \bar{v}_i &= & v_i - \frac{h}{2}\po \na_{x_i} U_{N}(\bfx) + \na_{x_i} U_{N}(\bar{\bfx}) \pf.
    \end{array}\right.
\end{equation}

\begin{lem}
For any $\varepsilon>0$ with $\varepsilon h \leqslant 1/4$,
\[\bfV\po\bar{\bfz}\pf    \leqslant (1+\varepsilon h ) \bfV(\bfz) + \frac{N h}{\varepsilon^5}\|D F\|_\infty^6   \,.\]
\end{lem}
\begin{proof}
Due to \eqref{eq:nablaUN}, $|\na_{x_i} U_N(\bfx)| \leqslant \|D F\|_\infty$ for all $\bfx$. By Jensen inequality, for any $a,b\geqslant 0$ and $p\in(0,1)$, 
\[(a+b)^6 = \po p \frac{a}{p} + (1-p) \frac{b}{1-p}\pf^6 \leqslant p^{-5} a^6 + (1-p)^{-5} b^6\]
For any $\varepsilon>0$, taking $p=(1+\varepsilon)^{-5}$ and using that $(1-(1+r)^{-5}) \geqslant r$ for $r\leqslant 1/4$ (simply by roughly lower-bounding the derivative), we get that
\[(a+b)^6  \leqslant (1+\varepsilon) a^6 + \frac{b^6}{\varepsilon^5} \]
for all $\varepsilon\in(0,1/4]$. Using this, we bound
\[\bfV\po\bar{\bfz}\pf  \leqslant \sum_{i=1}^N \po  |v_i| + \|D F\|_\infty h \pf ^6 \leqslant (1+\varepsilon h ) \bfV(\bfz) + \frac{N h}{\varepsilon^5}\|D F\|_\infty^6   \,.\]
for any $\varepsilon>0$ with $h\varepsilon \leqslant 1/4$, as desired.
\end{proof}

\begin{lem}\label{lem:E(G)}
For $w\in\R^d$ and $G$ a $d$-dimensional Gaussian variable, for any $\varepsilon\in(0,1/10]$,
\[  \mathbb E \po |\eta w + \sqrt{1-\eta^2} G|^6\pf\leqslant (1+\varepsilon ) \eta^6 |w|^6 + 87 \frac{(1-\eta^2)^3 d^3}{\varepsilon^2  } \,. \] 

\end{lem}

\begin{proof}
Expanding the squared norm,
\begin{eqnarray*}
 \lefteqn{ \mathbb E \po |\eta w + \sqrt{1-\eta^2} G|^6\pf}\\
  & = &  \mathbb E \po \po \eta^2 |w|^2 + 2 \eta \sqrt{1-\eta^2} w\cdot G + (1-\eta^2) |G|^2 \pf^3\pf \\
 & = & \eta^6 |w|^6 + (1-\eta^2)^3  \mathbb E(|G|^6) + 3 \eta^4 |w|^4(1-\eta^2) \mathbb E(|G|^2)  +  3 \eta^2 (1-\eta^2)^2 |w|^2 \mathbb E(|G|^4)\\
 & & + 12 (1-\eta^2)^2 \eta^2 \mathbb E \po |G|^2 (w\cdot G)^2\pf + 12(1-\eta^2)\eta^4 |w|^2 \mathbb E \po  (w\cdot G)^2\pf\,.
\end{eqnarray*}
Thanks to Jensen inequality,
\begin{equation}
    \label{eq:G6}
    \mathbb E(|G|^6) = d^3 \mathbb E\po \po \frac1d\sum_{i=1}^d |G_i|^ 2\pf^3\pf \leqslant d^3 \mathbb E (G_1^6) = 15 d^3\,,
\end{equation}
and similarly
\[\mathbb E(|G|^4) \leqslant 3 d^2\,.\]
By isotropy, $w\cdot G$ is a $1$-dimensional centered Gaussian variable with variance $|w|^2$. Using Cauchy-Schwarz inequality, we  bound
\[\mathbb E \po |G|^2 (w\cdot G)^2\pf \leqslant \sqrt{\mathbb E \po |G|^4\pf \mathbb E \po  (w\cdot G)^4\pf } \leqslant 3 d |w|^2 \,. \]
Gathering these bounds give 
\begin{eqnarray*}
 \lefteqn{ \mathbb E \po |\eta w + \sqrt{1-\eta^2} G|^6\pf}\\
  & \leqslant  & \eta^6 |w|^6 + 15 (1-\eta^2)^3  d^3  + 15 \eta^4 (1-\eta^2) |w|^4  d  +  (9d^2 + 36d) \eta^2 (1-\eta^2)^2 |w|^2 d^2\,.
\end{eqnarray*}
Now we bound the last terms involving $|w|^2$ or $|w|^4$ by using the inequality
\begin{equation} \label{eq:split}
    a^p b  \leqslant \frac{\varepsilon}2 a^3 +\po 1 - \frac{p}{3}\pf  b^{\frac{3}{3-p}}  \po  \frac{2 p}{3 \varepsilon}\pf^{\frac{p}{3-p}} \, ,
\end{equation}
which holds for any $a,b\geqslant 0$, $\varepsilon>0$ and $p\in\{1,2\}$ (proven by maximizing $a\mapsto a^p b -  \frac{\varepsilon}2 a^3 $ for a fixed $b$). Applying this with $a=\eta^2|w|^2$  gives 
\[  \mathbb E \po |\eta w + \sqrt{1-\eta^2} G|^6\pf\leqslant (1+\varepsilon ) \eta^6 |w|^6 + R \]
with
\begin{eqnarray*}
R &=& 15 (1-\eta^2)^3  d^3  + \frac13 \po 15   (1-\eta^2)    d \pf^3 \po  \frac{4}{3 \varepsilon }\pf^{2}    +  \frac23 \po (9d^2 + 36d)  (1-\eta^2)^2     \pf^{\frac32}  \po  \frac{2 }{3 \varepsilon }\pf^{\frac{1}{2}} \\ 
& \leqslant & 15+20+30\cdot\sqrt{3}  \frac{(1-\eta^2)^3 d^3}{\varepsilon^2  } \leqslant 87 \frac{(1-\eta^2)^3 d^3}{\varepsilon^2  }\,, 
\end{eqnarray*} 
where in the last line we assumed that $\varepsilon  \leqslant 1/10$ and used that $d\geqslant 1$ to simplify. This concludes the proof.
\end{proof}

Applying successively the two lemmas (with $\varepsilon$ replaced by $\varepsilon h$ in the second one), we get that for all $\varepsilon >0$ such that $\varepsilon h \leqslant 1/10$, 
\begin{eqnarray*}
\nu_{n+1} (\bfV) & \leqslant & (1+\varepsilon h)^2 \eta^6 \nu_n(\bfV) + (1+\varepsilon h)  87 N \frac{(1-\eta^2)^3 d^3}{\varepsilon^2 h^2   }  + \frac{N h}{\varepsilon^5}\|D F\|_\infty^6  \\
& \leqslant & (1+\varepsilon h)^2 (1-\gamma h)^6  \nu_n(\bfV) + 766 N \frac{\gamma ^3 d^3 h }{\varepsilon^2    }  + \frac{N h}{\varepsilon^5}\|D F\|_\infty^6
\end{eqnarray*}
where we used that $1-\eta^2 \leqslant 2 \gamma h$. Taking $\varepsilon=\gamma$ so that $(1+\varepsilon h)(1-\gamma h) \leqslant 1$, this gives 
\[ \nu_{n+1} (\bfV)\leqslant   (1-\gamma h)  \nu_n(\bfV) + Nh \co 766   \gamma  d^3    + \frac{\|D F\|_\infty^6  }{\gamma^5}\cf \,. \] 

\subsection{Euclidean space in continuous-time}\label{sec:LyapContinuous}

Before proceeding with the case in $\R^d$, let us first consider, as a warm-up, the same question but for the continuous-time Langevin diffusion. This will enable to identify more clearly (as, moreover, we won't try to make the constants explicit) the role of various terms that will also appear in the discrete-time settings. Let us first consider the case with no interaction at all.

\subsubsection{The non-interacting case}

We consider the continuous-time Langevin diffusion $(\mathbf{Z})_{t\ge 0}$ on $\R^{2d}$ with generator
\[L  = v\cdot\na_x  - \po \na V(x) + \gamma v\pf \cdot \na_v  + \gamma \Delta_v \,,\]
that we decompose as $L=L_1+ \gamma L_2$ with
\[L_1 = v\cdot\na_x  -   \na V(x)  \cdot \na_v  \,,\qquad L_2 = -   v \cdot \na_v + \Delta_v 
\]
and the Lyapunov function 
\[\bfV(z) = \po H_0(z) + G_0(z)\pf^3 \]
where 
\[H_0(x,v)  =  V(x) + \frac12 |v|^2\,,\qquad G_0(x,v) =  \alpha x\cdot v\,.\]
Then
\[\partial_t \mathbb E \po \bfV(Z_t) \pf = \mathbb E \po L \bfV(Z_t) \pf \,. \]
Using that $L$ is a diffusion generator,
\begin{equation}
\label{eq:LyapLangevinCont}
L\bfV = 3(H_0+G_0)^2 L(H_0+G_0) + 6 (H_0+G_0)\Gamma(H_0+G_0)\,,
\end{equation}
where $\Gamma(\varphi)=\gamma^2 |\na_v \varphi|^2$ is the associated carré du champ. At infinity, $(H_0+G_0)(z)$ scales like $|z|^2$  and its gradient scales like $|z|$. Hence, if we simply obtain that
\begin{equation}
\label{eq:LyapLangevinCont2}
L(H_0+G_0)(z) \leqslant - c|z|^2 + C
\end{equation}
for some $c,C>0$, we will have in~\eqref{eq:LyapLangevinCont}
 a negative term of order $|z|^6$  (like $\bfV$) and all other terms will be of order at most $|z|^4$. This will thus give
 \begin{equation}
\label{eq:LyapLangevinCont3}
L\bfV = - c' \bfV + C'
\end{equation}
for some $c',C'>0$, which is what we need. Hence, we can forget the power 3 and focus on establishing \eqref{eq:LyapLangevinCont2}. We treat separately the effect of $L_1$ and $L_2$ on respectively $H_0$ and $G_0$ to highlight the role of the different terms (since similar terms will play similar roles in the discrete-time case). The energy being preserved along the Hamiltonian dyanmics,
\begin{equation}
\label{eq:continuousLangevin1}
L_1 H_0 = 0\,,
\end{equation}
while, using~\eqref{condition:VLyap}, 
\begin{equation}
\label{eq:continuousLangevin2}
L_1G_0(z) =  \alpha \po |v|^2 - x\cdot \na V(x)\pf \leqslant \alpha \po |v|^2 - r|x|^2 + K d\pf  \,.
\end{equation} 
This term thus provides the negative term in $x$, at the cost of a positive term in $v$. This will be compensated by the stochastic part, since
\[L_2 H_0(z) = -  |v|^2 + d \]
while
\[L_2G_0(z) = - \alpha x\cdot v \leqslant \frac12 |v|^2 + \frac{\alpha^2}2 |x|^2 \,.  \] 
Gathering all these computations,
\begin{align}
L\po H_0+G_0\pf(z) &\leqslant  \alpha \po |v|^2 - r|x|^2 + K d\pf  + \gamma \po -\frac12   |v|^2 + \frac{\alpha^2}2 |x|^2  + d\pf\nonumber   \\
& \leqslant -\frac{r \alpha}{2}|x|^2 - \frac{\gamma}{4}|v|^2 +  \alpha K d + \gamma d\label{eq:continuousLangevinbound}
\end{align}
by taking $\alpha = (r/\gamma)\wedge (\gamma/4)$. Inequality~\eqref{eq:LyapLangevinCont2} then follows from the equivalence at infinity between $H_0+G_0$ and $|z|^2$.

\subsubsection{The interacting case}\label{sec:LyapunovContinuousInteracting}

Adding the interaction to the previous computations, we get that
\[\partial_t \mathbb E \po \bfV(Z_t^i) \pf = \mathbb E \po L\bfV(Z_t^i) + D F\po X_t^i,\pi_{\mathbf{X}_t}\pf \cdot \na_{v_i} \bfV(Z_t^i) \pf  \]
The first term has been treated with~\eqref{eq:LyapLangevinCont3}. Using that $\na_{v_i} \bfV$ scales like $|z|^5$, the second one is bounded thanks to~\eqref{condition:lambdaM} as
\[|D F\po x_i,\pi_{\bfx}\pf \cdot \na_{v_i} \bfV(z_i)| \leqslant (M \sqrt{d} + \lambda |x_i| + \lambda m\po\bfx\pf ) C(|z_i|^5+1)\]
for some $C>0$, where we write
\begin{equation*}
    m(\bfx) = \frac1N \sum_{i=1}^N |x_i|\,.
\end{equation*}
The term $M|z_i|^5$ is negligible with respect to $\bfV(z_i)$, the term $|x_i||z_i|^5$ is of the same order and we treat it by assuming that $\lambda$ is small enough, and the term with $m(\bfx)$ (which also has a factor $\lambda$) is treated by using that
\[\frac1N \sum_{j=1}^N |x_j| |z_i|^5 \leqslant \frac CN \sum_{j=1}^N \po  |x_j|^6 + |z_i|^6\pf\]
for some $C>0$. As a conclusion, assuming that $\lambda$ is small enough, we end up with 
\[\partial_t \mathbb E \po \bfV(Z_t^i) \pf = -c'' \mathbb E \po \bfV(Z_t^i) +  C'' \frac{\lambda}N \sum_{j=1}^N |X_t^j|^6  \pf  + C'' \]
for some $c'',C''>0$ (independent from $\lambda$). Summing these inequalities over $i\in\cco 1,N\ccf$ and assuming again that $\lambda$ is small enough (and using that $\bfV(z)$ scales like $|z|^6$) we finally obtain
\[\partial_t \mathbb E \po \sum_{i=1}^N \bfV(Z_t^i) \pf = -c \mathbb E \po \sum_{i=1}^N  \bfV(Z_t^i)    \pf  + C N \,.  \]

\subsection{Euclidean space in discrete time}\label{sec:LyapDiscrete}

In this section, devoted to the proof of Proposition~\ref{prop:LyapunovRd}, $\X=\R^d$, Assumption~\ref{assu:LyapunovRd} holds and $\bfV$ is given by~\eqref{def:LyapRd}.

 As seen in~\eqref{eq:choicealpha}, the parameter $\alpha$ is chosen as a function of $\mathfrak{p}$, while $\lambda$ and $h$, as discussed, are bounded in Proposition~\ref{prop:LyapunovRd} by some function of $\mathfrak{p}$. For this reason, we can bound any expression depending on $\lambda,h$ in a non-decreasing manner  and on $\mathfrak{p},\alpha$ by some $C\in\mathcal C(\mathfrak{p})$, which we will do repeatedly in the proof of the next result.

 In order to make some sense out of the computations, at some places in the proof we will use a color code to describe the roles of various terms. The dominant terms (in terms of $h$ and $|z|$) will be written in \rouge{red}. They have to be particularly well controlled and balanced one with another as they will provide the desired negative term, similarly to the continuous-time case in~\eqref{eq:continuousLangevinbound}. The terms in \magenta{magenta} are those which are of the same order as the red ones but come from the linearly-growing part of the interacting force in~\eqref{condition:lambdaM}, in other words they are in factor of $\lambda$. They are treated at the end by assuming that $\lambda$ is small enough so that these terms are controlled by the red ones, exactly as presented in the continuous-time case in Section~\ref{sec:LyapunovContinuousInteracting}. Finally, the terms in \bleu{blue} are those which are of higher order in $h$ or lower order in $|z|$ than the red terms, and thus they will be controlled by taking $h$ small enough at the end or they will simply  contribute to the constant term in the final inequality. We will not keep track of these terms, which will systematically be roughly bounded using some constant $C\in\mathcal C(\mathfrak{p})$.
 
First, we  analyze the effect of the Verlet step on  $\varphi$, which is the analogue in the discrete time case of the computations~\eqref{eq:continuousLangevin1} and \eqref{eq:continuousLangevin2}. As in Section~\ref{sec:proof-Lyap-torus}, we write $\bar{\bfz}=(\bar{\bfx},\bar{\bfv})$ the position after one Verlet step starting from $\bfz=(\bfx,\bfv)\in\R^{2dN}$, given by~\eqref{eq:Verletxivi}.

\begin{lem}\label{lem:Lyap1}
Under Assumption~\ref{assu:LyapunovRd}, assuming furthermore that $3\lambda < r$, there exists a constant $C\in\mathcal C(\mathfrak{p})$  such that, for any $N\in\N$, $\bfz\in\R^{2dN}$ and $i\in\cco 1,N\ccf$,
\begin{multline*}
\varphi(\bar{z}_i)   \leqslant  \varphi(z_i) + \rouge{h \po -  \frac{r\alpha }4   |x_i|^2 + \po \frac{5}{4}\alpha+ \frac{\lambda}2\pf |v_i|^2\pf} + \magenta{\frac{h \lambda}2(1+\alpha)m^2(\bfx)} \\
+ \bleu{ hd \po \alpha K + \frac{\alpha M^2}{r} + \frac{M^2}{\alpha}\pf  + C h^2 \po |z_i|^2 + m^2(\bfz)+d\pf  }\,.
\end{multline*}
\end{lem}

\begin{proof}
Since $\na_{x_i} U_N(\bfx) = DF(x_i,\pi_{\bfx})$, \eqref{condition:lambdaM} and~\eqref{cond:VlyapGradient} imply
\begin{equation}
\label{eq:boundnablaUn}
|\na_{x_i} U_{N}(\bfx) | \leqslant (L+\lambda) |x_i| + M\sqrt{d}  +\lambda m(\bfx)\,,
\end{equation}
where we recall the notation $m(\bfx) = \frac1N \sum_{i=1}^N |x_i|$. 

First, consider the potential energy:
\begin{eqnarray*}
V(\bar{x_i})  & = & V\po x_i + h v_i - \frac{h^2}{2} \na_{x_i} U_{N}(\bfx) \pf  \\
& \leqslant & V(x_i) + \po h v_i - \frac{h^2}{2} \na_{x_i} U_{N}(\bfx)\pf \cdot \na V(x_i) + \frac12 L \left|h v_i - \frac{h^2}{2} \na_{x_i} U_{N}(\bfx)\right|^2 \\
& \leqslant & V(x_i) + \rouge{h v_i \cdot \na V(x_i)} + \bleu{\frac{h^2}{2} L|x_i||\na_{x_i} U_{N}(\bfx) |  + Lh^2\po |v_i|^2 + \frac{h^2}{4}|\na_{x_i} U_{N}(\bfx) |^2\pf}  
\end{eqnarray*}
Denoting by $\bleu{B_1}$ the blue terms, using~\eqref{eq:boundnablaUn} we bound them by
\begin{eqnarray*}
\bleu{B_1} & \leqslant & h^2 \po \frac{L}2 (L+\lambda) |x_i|^2 + \frac{L}2 M\sqrt{d} |x_i| +\frac{L}2 \lambda |x_i|m(\bfx) + L|v_i|^2 + \frac34\po (L+\lambda)^2 |x_i|^2 + M^2 d +\lambda^2 m^2(\bfx) \pf  \pf \\
& \leqslant & C_1 h^2 \po |x_i|^2 + |v_i|^2 + m^2(\bfx) + d\pf
\end{eqnarray*}
for some explicit $C_1\in\mathcal C(\mathfrak{p})$ depending only $L,M,\lambda$. Plugging this in the previous inequality gives
\begin{equation}
\label{eq:Vbarxi}
V(\bar{x_i})   \leqslant  V(x_i) + \rouge{h v_i \cdot \na V(x_i)} + \bleu{C_1 h^2 \po |x_i|^2 + |v_i|^2 + m^2(\bfx) + d\pf}\,.  
\end{equation}

Next, bound the kinetic energy as:
\begin{eqnarray}
\frac12 |\bar{v_i}|^2  & = & \frac{1}{2}\left | v_i - \frac{h}{2}\po \na_{x_i} U_{N}(\bfx) + \na_{x_i} U_{N}(\bar{\bfx}) \pf\right|^2\nonumber \\
& = & \frac12|v_i|^2 - \frac h2 v_i\cdot  \po \na_{x_i} U_{N}(\bfx) + \na_{x_i} U_{N}(\bar{\bfx}) \pf + \frac{h^2}8 \left|  \na_{x_i} U_{N}(\bfx) + \na_{x_i} U_{N}(\bar{\bfx}) \right|^2 \nonumber \\
& \leqslant & \frac12|v_i|^2  \ \rouge{ - \ h v_i\cdot \na V(x_i)} + \magenta{ \lambda \frac h2 |v_i| \po |x_i|+|\bar{x}_i| + m(\bfx)+m(\bar{\bfx})\pf } \nonumber\\
& & \ + \bleu{\frac{h}{2}L|v_i||\bar{x}_i-x_i|+ hM \sqrt{d} |v_i| + \frac{h^2}8 \left|  \na_{x_i} U_{N}(\bfx) + \na_{x_i} U_{N}(\bar{\bfx}) \right|^2}\,.\label{eq:barvi2}
\end{eqnarray}
In the magenta terms, we will further bound
\begin{eqnarray}
|\bar{x}_i| &\leqslant &   |x_i| + h |v_i| + \frac{h^2}{2}|\na_{x_i} U_{N}(\bfx) |\nonumber \\
&\leqslant &  \magenta{|x_i|} + \bleu{h |v_i| + \frac{h^2}{2}\po (L+\lambda) |x_i| + M\sqrt{d} +\lambda m(\bfx)\pf } \label{eq:barxi}
\end{eqnarray}
and similarly 
\begin{eqnarray}
m(\bar{\bfx})  &\leqslant &  m(\bfx) + h m(\bfv) + \frac{h^2}{2N} \sum_{j=1}^N |\na_{x_j} U_{N}(\bfx) | \nonumber \\
& \leqslant &   m(\bfx) + h m(\bfv) + \frac{h^2}{2N} \sum_{j=1}^N \co (L+\lambda) |x_j| + M\sqrt{d}  +\lambda m(\bfx)\cf \nonumber  \\
& =  & \magenta{m(\bfx)} + \bleu{\frac{h^2}2 \po L+2\lambda\pf    m(\bfx) + h m(\bfv) + \frac{h^2 M \sqrt{d}}{2}} \label{eq:mbarxi}\,.
\end{eqnarray}
To control the blue terms of~\eqref{eq:barvi2}, using~\eqref{eq:barxi} and \eqref{eq:mbarxi} we bound
\begin{eqnarray*}
    |v_i||\bar{x}_i-x_i|\le h|v_i|^2+ \frac{h^2}{2}|v_i||\nabla_{x_i} U_N(\bfx)|\le 2 h |v_i|^2+ (L+\lambda)^2\frac{h^3}{4}|x_i|+ \frac{h^3}{4}M^2d + \lambda \frac{h^3}{4}m^2(\bfx)
\end{eqnarray*}
and
\begin{eqnarray*}
\left|  \na_{x_i} U_{N}(\bfx) + \na_{x_i} U_{N}(\bar{\bfx}) \right| & \leqslant & (L+\lambda) \po |x_i| + |\bar{x_i}|\pf  + 2M \sqrt{d}+\lambda \po m(\bfx)+ m(\bar{\bfx})\pf   \\
& \leqslant & C \po |x_i|+|v_i| + m(\bfx)+m(\bfv)+\sqrt{d}\pf 
\end{eqnarray*}
for some $C\in\mathcal C(\mathfrak{p})$ depending increasingly on $L,M,\lambda,h$. Plugging these bounds and \eqref{eq:barxi} and \eqref{eq:mbarxi} in~\eqref{eq:barvi2} and treating the blue terms of order $h^2$ as we did for $B_1$, we end up with 
\begin{multline}
\frac12 |\bar{v_i}|^2  
 \leqslant  \frac12|v_i|^2  \ \rouge{ - \ h v_i\cdot \na V(x_i)} + \magenta{ \lambda h |v_i| \po |x_i| + m(\bfx)\pf }  \\
 + \bleu{hM \sqrt{d} |v_i| + C_2 h^2 \po |x_i|^2+|v_i|^2 + m^2(\bfx)+m^2(\bfv)+d\pf}\label{eq:barvi2prime}
\end{multline}
for some $C_2\in\mathcal C(\mathfrak{p})$ depending increasingly on $L,M,\lambda,h$.
 
Finally, consider the crossed term:
\begin{eqnarray*}
\bar{x_i}\cdot \bar{v_i} &  =&   \po x_i + h v_i - \frac{h^2}{2} \na_{x_i} U_{N}(\bfx) \pf \cdot\po v_i - \frac{h}{2}\po \na_{x_i} U_{N}(\bfx) + \na_{x_i} U_{N}(\bar{\bfx}) \pf \pf    \\
& \leqslant  & x_i\cdot v_i + \rouge{h \po - x_i\cdot \na V(x_i) + |v_i|^2\pf} + h |x_i| \po \bleu{M\sqrt{d}} + \magenta{ \frac{\lambda}2(|x_i|+|\bar{x}_i|+m(\bfx)+m(\bar{\bfx}))}\pf  \\
& & \ + \bleu{ \frac{h}{2} \left|h v_i - \frac{h^2}{2} \na_{x_i} U_{N}(\bfx) \right| | \na_{x_i} U_{N}(\bfx) + \na_{x_i} U_{N}(\bar{\bfx}) |  + \frac{h^2}{2} |v_i||\na_{x_i} U_{N}(\bfx)|}
\\ & & \ \bleu{+ Lh|x_i||\bar{x}_i-x_i| \,. }
\end{eqnarray*}
The magenta terms are again simplified further by using \eqref{eq:barxi} and \eqref{eq:mbarxi} and the blue terms are treated as before, which leads to
\begin{multline*}
\bar{x_i}\cdot \bar{v_i} \leqslant x_i\cdot v_i + \rouge{h \po - x_i\cdot \na V(x_i) + |v_i|^2\pf} + \magenta{h \lambda |x_i|  (|x_i|+m(\bfx))} \\
+ \bleu{hM\sqrt{d}|x_i| + C_3 h^2 \po |x_i|^2+|v_i|^2 + m^2(\bfx)+m^2(\bfv)+d\pf } \,.
\end{multline*}
Combining this inequality with~\eqref{eq:Vbarxi} and~\eqref{eq:barvi2prime} and using the assumption gives
\begin{multline*}
\varphi(\bar{z}_i)   \leqslant  \varphi(z_i) +  \rouge{h\alpha \po - r|x_i|^2 + |v_i|^2\pf} + \magenta{h \lambda(|v_i|+\alpha |x_i| ) (|x_i|+m(\bfx))} \\
+ \bleu{ h(\alpha K d +\alpha M\sqrt{d} |x_i|+ M \sqrt{d}|v_i|) + C_4 h^2 \po |x_i|^2+|v_i|^2 + m^2(\bfx)+m^2(\bfv)+d\pf  }\,.
\end{multline*}
To get exactly the claimed result, it only remains, first, to bound the magenta terms as
\[  (|v_i|+\alpha |x_i| ) (|x_i|+m(\bfx)) \leqslant \frac{3\alpha}{2}|x_i|^2 + \frac{1+\alpha}{2}m^2(\bfx) + \frac12|v_i|^2  \]
and use that $3\lambda < r$ and, second, to bound in the blue terms of order $h$
\[\alpha M\sqrt{d}|x_i| \leqslant \frac{\alpha M^2 d}{r} + \frac{r\alpha}{4}|x_i|^2 \qquad \text{and} \qquad M\sqrt{d}|v_i| \leqslant \frac{ M^2 d}{\alpha} + \frac{\alpha}{4}|v_i|^2\,. \] 

\end{proof}

Next, we consider the effect of the Verlet step on $\varphi^3$.

\begin{lem}
Under Assumption~\ref{assu:LyapunovRd}, assuming furthermore that,
\[ \lambda   \leqslant  \min\po \frac{r}{3}, \frac{2\alpha}{3}, \frac{ r\alpha c_0^2}{176(1+\alpha)^3} \pf \,,\]
there exists $C\in\mathcal C(\mathfrak{p})$ such that 
for any $N\in\N$ and $\bfz\in\R^{2dN}$,
\[\bfV(\bar{\bfz}) \leqslant  \po 1 + \magenta{h\lambda}\pf \bfV(\bfz)  + \rouge{ h \sum_{i=1}^N \co  -  \frac{r\alpha c_0^2}{11}   |x_i|^6 + \alpha A' |v_i|^6 \cf }     
  + \bleu{ C   \co  h N d^3
 + h^2  \bfV(z) \cf }\,,\]
 where $A'$ is given in \eqref{eq:A'} below.
\end{lem}

\begin{proof}
Notice that, using~\eqref{eq:boundphi} and Lemma~\ref{lem:Lyap1}, there exists a constant $C\in\mathcal C(\mathfrak{p})$ such that for all $i\in[[1,N]]$, $\varphi(z_i) \leqslant C(|z_i|^2+d)$ and $|\varphi(\bar{z}_i) - \varphi(z_i)| \leqslant C h(|z_i|^2 +m^2(\bfz)+d)$. Then
\begin{align*}
\varphi^3(\bar{z}_i) &= \varphi^3(z_i) + 3 (\varphi(\bar z_i)-\varphi(z_i)) \varphi^2(z_i) + 3 (\varphi(\bar z_i)-\varphi(z_i))^2 \varphi(z_i)  +  (\varphi(\bar z_i)-\varphi(z_i))^3  \\
 & \leqslant  \varphi^3(z_i) + 3 (\varphi(\bar z_i)-\varphi(z_i)) \varphi^2(z_i) + \bleu{ C' h^2 (|z_i|^6 +m^6(\bfz)+d^3)}\,,
\end{align*}
for some $C'\in\mathcal C(\mathfrak{p})$. Using again Lemma~\ref{lem:Lyap1} and that $\varphi(z_i) \leqslant C(|z_i|^2+d)$ to deal with the blue terms give
\begin{multline}
\varphi^3(\bar{z}_i) 
  \leqslant  \varphi^3(z_i) + 3 h \co \rouge{  -  \frac{r\alpha }4   |x_i|^2 + \po \frac{5}{4}\alpha+ \frac{\lambda}2\pf |v_i|^2} + \magenta{\frac{\lambda}2(1+\alpha)m^2(\bfx)}\cf  \varphi^2(z_i) \\
 + \bleu{ \ C''  \co  h d^3
 + h^2 (|z_i|^6 +m^6(\bfz))\cf }\,,\label{eq:varphi3}
\end{multline}
for some $C''\in\mathcal C(\mathfrak{p})$.  For simplicity, we use~\eqref{eq:boundphi} again to get
 \begin{equation}
 \label{eq:xivarphi}
|x_i|^2  \varphi^2(z_i) \geqslant \frac{c_0^2}{4}|x_i|^6\,.
 \end{equation}
and
\begin{align}
|v_i|^2 \varphi^2(z_i) &  \leqslant |v_i|^2 \po dR_1+ \frac{3c_1}{2}  |x_i|^2 + \frac{3}{4}|v_i|^2 \pf ^2 \nonumber \\
& \leqslant 3 |v_i|^2 \po d^2R_1^2+ \frac{9 c_1^2}{4} |x_i|^4 + \frac{9}{16}|v_i|^4 \pf \nonumber \\
& \leqslant \frac{r\alpha c_0^2 }{32 \po \frac{5}{4}\alpha+ \frac{\lambda}2\pf} |x_i|^6  + A |v_i|^6 + C d^6\label{eq:vivarphi}
\end{align}
for some $C\in\mathcal C(\mathfrak{p})$, where we used that by \eqref{eq:split}
\[ \frac{27 c_1^2}{4} |v_i|^2   |x_i|^4 \leqslant \frac{r\alpha c_0^2 }{32 \po \frac{5}{4}\alpha+ \frac{\lambda}2\pf} |x_i|^6  +  \frac{1}{3} \po \frac{27 c_1^2}{4} |v_i|^2 \pf^{3}  \po  \frac{ 64\po \frac{5}{4}\alpha+ \frac{\lambda}2\pf}{3 r\alpha c_0^2 }\pf^{2} \,,   \]
and thus set
\[A= \po 2 + 46656 \frac{c_1^6(\frac{5}{4}\alpha+\lambda)^2 }{r^2\alpha^2c_0^2 } \pf\,.\] 
Using this inequality together with \eqref{eq:xivarphi} in~\eqref{eq:varphi3} gives, using furthermore that $\lambda \leqslant 2\alpha/3$,
\begin{multline*}
\varphi^3(\bar{z}_i) 
  \leqslant  \varphi^3(z_i) +  h \co \rouge{  -  \frac{3 r\alpha c_0^2}{32}   |x_i|^6 + 5\alpha A |v_i|^6} + \magenta{\frac{3\lambda}2(1+\alpha)m^2(\bfx)} \varphi^2(z_i)\cf    \\
 + \bleu{ \ C'''  \co  h d^3
 + h^2 (|z_i|^6 +m^6(\bfz))\cf }\,,
\end{multline*}
with some $C'''\in\mathcal C(\mathfrak{p})$. Notice that the condition $\lambda \leqslant 2\alpha/3$ implies that
\begin{equation} \label{eq:A'}
     5A \leqslant \po 10 + 584820 \frac{c_1^6 }{r^2c_0^2 } \pf =: A' \,. 
\end{equation}
It remains to bound the magenta term using that
\[
\frac{3\lambda}2(1+\alpha)m^2(\bfx) \varphi^2(z_i)  \leqslant  \lambda \po \varphi^3(z_i) + \frac{108(1+\alpha)^3}{216  }  m^6(\bfx)   \pf \,,  
\]  
to get
\begin{multline*}
\varphi^3(\bar{z}_i) 
  \leqslant  \po 1 + \magenta{h\lambda}\pf \varphi^3(z_i) +  h \co \rouge{  -  \frac{3 r\alpha c_0^2}{32}   |x_i|^6 + \alpha A' |v_i|^6} + \magenta{\frac{\lambda(1+\alpha)^3 }2m^6(\bfx)} \cf    \\
 + \bleu{ \ C'''  \co  h d^3
 + h^2 (|z_i|^6 +m^6(\bfz))\cf }\,.
\end{multline*}
Summing these inequalities over $i\in\cco 1,N\ccf$ and using that by Jensen inequality, $m^6(\bfx) \leqslant \frac1N \sum_{j=1}^N |x_j|^6=:m_6(\bfx)$,  we get
\begin{multline*}
\bfV(\bar{\bfz}) \leqslant   \po 1 + \magenta{h\lambda}\pf \bfV(\bfz)  +  h \co \rouge{   -  \frac{3 r\alpha c_0^2}{32}   N m_6(\bfx) + \alpha A' Nm_6(\bfv)} + \magenta{\frac{\lambda(1+\alpha)^3 }2 Nm^6(\bfx)} \cf    \\
  + \bleu{ \ C'''  \co  h N d^3
 + h^2  \bfV(z) \cf }\\
 \leqslant  \po 1 + \magenta{h\lambda}\pf \bfV(\bfz)  +  h \co \rouge{   -  \frac{r\alpha c_0^2}{11}   N m_6(\bfx) + \alpha A' Nm_6(\bfv)} \cf    
  + \bleu{ \ C'''  \co  h N d^3
 + h^2  \bfV(z) \cf }
\end{multline*}
where we used in the last line that 
\[ \lambda   \leqslant  \frac{ r\alpha c_0^2}{176(1+\alpha)^3} \,.\]
This concludes the proof of the lemma.

\end{proof}

It remains to take into account the stochastic part, which requires to generalize Lemma~\ref{lem:E(G)} as follows:

\begin{lem} \label{lem:Gaussvar}
For $a,c\geqslant 0$,  $b\in\R^d$ and $G$ a $d$-dimensional Gaussian variable, 
\[  \mathbb E \po \po a + b \cdot G + c |G|^2 \pf^3\pf  \leqslant a^3 + 15d^3c^3+3a^2 cd + 9ac^2 d^2 +9cd|b|^2+3 a |b|^2\,.  \]

\end{lem}

\begin{proof}
Expanding the power and using $\mathbb E[b\cdot G]=\mathbb E[(b\cdot G)^3]=\mathbb E[(b\cdot G)|G|^2]=0$,
\begin{eqnarray*}
\mathbb E \po \po a + b \cdot G + c |G|^2 \pf^3\pf 
 & = & a^3 + c^3 \mathbb E(|G|^6) + 3 a^2 c \mathbb E(|G|^2) + 3 a c^2 \mathbb E(|G|^4) \\
 & & + 3   c \mathbb E \po |G|^2 (b\cdot G)^2\pf + 3 a   \mathbb E \po  (b\cdot G)^2\pf\,.
\end{eqnarray*}
Thanks to Jensen inequality,
\[\mathbb E(|G|^6) = d^3 \mathbb E\po \po \frac1d\sum_{i=1}^d |G_i|^ 2\pf^3\pf \leqslant d^3 \mathbb E (G_1^6) = 15 d^3\,,\]
and similarly
\[\mathbb E(|G|^4) \leqslant 3 d^2\,.\]
By isotropy, $b\cdot G$ is a $1$-dimensional centered Gaussian variable with variance $|b|^2$. Using Cauchy-Schwarz inequality, we  bound
\[\mathbb E \po |G|^2 (b\cdot G)^2\pf \leqslant \sqrt{\mathbb E \po |G|^4\pf \mathbb E \po  (b\cdot G)^4\pf } \leqslant 3 d |b|^2 \,. \]
Gathering these bounds gives
\begin{eqnarray*}
\mathbb E \po \po a + b \cdot G + c |G|^2 \pf^3\pf   & \leqslant  & a^3 + 15 d^3 c^3   + 3 a^2 c d   + 9 a c^2 d^2   
  + 9   c |b|^2 d  + 3 a |b|^2   \,.
\end{eqnarray*}
\end{proof}

\begin{cor}\label{cor:LyapStoch}
There exists $C\in\mathcal C(\mathfrak{p})$ such that, for any $z_i=(x_i,v_i)\in\R^{2d}$, denoting $W= \eta v_i + \sqrt{1-\eta^2} G$,   for any $\varepsilon_1>0,\varepsilon_2 \in(0,1]$,
\begin{eqnarray}
\mathbb E \po \varphi^3(x_i,W)\pf & \leqslant & \varphi^3(z_i)  + 3 h \varphi^2(z_i) \co \bleu{ \varepsilon_2  \varphi(z_i)} + \rouge{ \varepsilon_1 \alpha \gamma  |x_i|^2 + \po    \frac{\alpha \gamma }{4\varepsilon_1}-\gamma \pf  |v_i|^2 } \cf\nonumber\\
 &  & \ +  \bleu{C h^2 \varphi^3(z_i) + \frac{C}{\varepsilon_2^2} h d^3} \label{eq:Evarphi1}\\
\mathbb E \po |W|^2\varphi^2(x_i,W)\pf & \leqslant&     |v_i|^2 \varphi^2(z_i) + \bleu{C h \varphi^3(z_i)  } \label{eq:Evarphi2} \\
\mathbb E \po \varphi^2(x_i,W)\pf & \leqslant & \po 1  + \bleu{Ch} \pf \varphi^2(z_i) \,.\label{eq:Evarphi3}
\end{eqnarray}
\end{cor}
\begin{proof}
We only give the details for the first inequality, the arguments for the two other being similar (and simpler as we don't make explicit the terms of order $h$). Expanding the squared norm,
\begin{align*}
\varphi(x_i,W)=  V_0(x_i) + \frac12 |W|^2 + \alpha x_i\cdot W 
&= a + b\cdot G + c |G|^2
\end{align*}
with 
\[a=   V_0(x_i) + \frac12 \eta^2 |v_i|^2 + \alpha x_i\cdot   \eta v_i\,,\qquad b=  \sqrt{1-\eta^2} \po \eta  v_i - \alpha x_i\pf \,,\qquad c= \frac{1-\eta^2}2\,. \]
Then by Lemma~\ref{lem:Gaussvar},
\begin{equation}
\label{eq:Evarphiabc}
\mathbb E \po \varphi^3(x_i,W)\pf \leqslant a^3    + 3 a^2 c d + 15 d^3 c^3    + 9 a c^2 d^2   
  + 9   c |b|^2 d  + 3 a |b|^2  \,.
  \end{equation}
  We bound
  \begin{align*}
  a & \leqslant \varphi(z_i)+ \varepsilon_1 \alpha (1-\eta) |x_i|^2 + \po    \frac{\alpha(1-\eta)}{4\varepsilon_1} +  \frac{\eta^2 - 1}2\pf  |v_i|^2 \\
  |b|^2 &\leqslant (1-\eta^2)\po |v_i|+\alpha |x_i|\pf^2\,.
  \end{align*}
  Using that $1-\eta = \gamma h$, the four last terms of \eqref{eq:Evarphiabc} are thus bounded  (taking e.g. $\varepsilon_1=1$ in the previous bound on $a$) as
  \[ 15 d^3 c^3    + 9 a c^2 d^2   
  + 9   c |b|^2 d  + 3 a |b|^2 \leqslant \bleu{C h^2 \varphi^3(z_i) + Ch \varphi^2(z_i)}\,,\]
  and, developing $a^3=(a-\varphi(z_i)+\varphi(z_i))^3$,
  \begin{align*}
 a^3 
  & \leqslant \varphi^3 (z_i)+ 3 \varphi^2(z_i) \co a - \varphi(z_i)\cf  + \bleu{C h^2 \varphi^3(z_i)} \\
 & \leqslant \varphi^3 (z_i)+ 3 \varphi^2(z_i)\co  \varepsilon_1 \alpha \gamma h |x_i|^2 + \po    \frac{\alpha \gamma h}{4\varepsilon_1}-\gamma h\pf  |v_i|^2  \cf + \bleu{C
h^2\varphi^3(z_i)}\,.  
  \end{align*}
We bound the remaining term of \eqref{eq:Evarphiabc} as
\[
3 a^2 c d  \leqslant  \varepsilon_2 h a^3 +  \frac{1}{3}   (3cd)^{3}  \po  \frac{ 2}{3 \varepsilon_2 h}\pf^{2} \leqslant   \varepsilon_2 h a^3 + \bleu{\frac{C}{\varepsilon_2^2} h d^3} \]
so that \eqref{eq:Evarphiabc}  gives
\begin{eqnarray*}
\mathbb E \po \varphi(x_i,W)\pf & \leqslant & (1+\varepsilon_2 h ) a^3  +  \bleu{C h^2 \varphi^3(z_i) + \frac{C}{\varepsilon_2^2} h d^3 + C h \varphi^2(z_i)}   \\
 & \leqslant & (1+\varepsilon_2 h ) \varphi^3(z_i)  + 3 \varphi^2(z_i)\co  \varepsilon_1 \alpha \gamma h |x_i|^2 + \po    \frac{\alpha \gamma h}{4\varepsilon_1}-\gamma h\pf  |v_i|^2  \cf\\
 & & \ +  \bleu{C h^2 \varphi^3(z_i) + \frac{C}{\varepsilon_2^2} h d^3 + C h \varphi^2(z_i)} \,.  
\end{eqnarray*}
For the last blue term we use that $\varphi^2(z_i) \leqslant \varepsilon_2 \varphi^3(z_i) + \frac{C}{\varepsilon_2^2} d^3$. This concludes the proof of~\eqref{eq:Evarphi1}.
\end{proof}

  \begin{proof}[Proof of Proposition~\ref{prop:LyapunovRd}]
 Applying the velocity randomization step to~\eqref{eq:varphi3} gives
 \begin{multline*}
 \mathcal P\po \varphi^3\pf (z_i) 
  \leqslant  \mathbb E\po \varphi^3(x_i,W_i)\pf  + 3 h \co \rouge{  -  \frac{r\alpha }4   |x_i|^2 } + \magenta{\frac{\lambda}2(1+\alpha)m^2(\bfx)}\cf  \mathbb E\po \varphi^2(x_i,W_i)\pf \\
  + 3 h \rouge{   \po \frac{5}{4}\alpha+ \frac{\lambda}2\pf \mathbb E\po |W_i|^2  \varphi^2(x_i,W)\pf }    + \bleu{ \ C''  \co  h d^3
 + h^2 (|x_i|^6 +\mathbb E \po |W_i|^6 +  m^6(\bfx,\mathbf{W})\pf) \cf }\,.
\end{multline*}
The expectations are bounded using Corollary~\ref{cor:LyapStoch} and Lemma~\eqref{lem:E(G)}. All resulting terms of order $h^2$ ore more are roughly bounded and incorporated in the blue terms. We end up with 
 \begin{multline*}
 \mathcal P\po \varphi^3\pf (z_i) 
  \leqslant  \varphi^3(z_i)  + 3 h \Big[ \bleu{ \varepsilon_2  \varphi(z_i)} + \rouge{ \alpha \po \varepsilon_1  \gamma -  \frac{r }4 \pf   |x_i|^2 + \po    \alpha \po \frac{ \gamma }{4\varepsilon_1} + \frac{19}{12}\pf -\gamma \pf  |v_i|^2 } \\
   + \magenta{\frac{\lambda}2(1+\alpha)m^2(\bfx)}\Big]  \varphi^2(z_i) 
   + \bleu{ \ C''  \co  h d^3
 + h^2 (|z_i|^6 + m^6(\bfz)) \cf }\,.
\end{multline*}
Notice that the only difference with respect to~\eqref{eq:varphi3} is the order $h$ term from~\eqref{eq:Evarphi1} appearing from the term of order $0$ (in $h$) in~\eqref{eq:varphi3}. In other words, since we are only working with the linear terms of order $h$, naturally, the effect of the Hamiltonian part and of the velocity randomization parts sum up, as in the continuous case~\eqref{eq:LyapLangevinCont}.

Now, in order for the dominant red terms to be negative, we take 
\[\varepsilon_1 = \frac{r}{8\gamma}\,,\qquad \alpha = \min \po \frac{\gamma}{2} \po \frac{ \gamma }{4\varepsilon_1} + \frac{19}{12}\pf^{-1}, \sqrt{\frac{c_0}2} \pf\,.  \] 
Thanks to~\eqref{eq:boundphi}, the red terms can thus be bounded as
\begin{align*}
-  3 \po  \frac{\alpha r }8    |x_i|^2  + \frac{\gamma}{2}   |v_i|^2\pf \varphi^2(z_i) &\leqslant - \min\po \frac{\alpha r }{4c_1}, 2\gamma\pf   \po   \varphi(z_i)- dR_1 \pf \varphi^2(z_i) \\
&\leqslant - 2\theta  \varphi^3(z_i) + C d^3\,.  
\end{align*}
with $2\theta = \min\po \frac{\alpha r }{5c_1}, \gamma\pf $.

The magenta terms are treated as
\[m^2(\bfx) \varphi^2(z_i) \leqslant \frac{2}{3}m^6(\bfx) + \frac{2}{3} \varphi^3(z_i) \leqslant \frac{16}{3c_0^3 N} \bfV(\bfz) + \frac{2}{3} \varphi^3(z_i) \,, \]
where we used Jensen inequality and \eqref{eq:boundphi} to control $m^6(\bfx)$.
 
 We sum up the resulting inequality over $i\in\cco 1,N\ccf$, obtaining
 \[
 \mathcal P(\bfV)(\bfz) 
  \leqslant  \po 1+ h\co \bleu{\varepsilon_2+\hat Ch} +  \magenta{    \frac{\lambda}2(1+\alpha)\po \frac{16}{3c_0^3 N}  + \frac{2}{3}\pf  } \rouge{-2\theta}\cf  \pf  \bfV(\bfz) 
   + \bleu{ Chd^3   }\,,
\]
for $C,\hat C\in\mathcal C(\mathfrak{p})$. It remains to take $\varepsilon_2=\theta/3$, to set $h_0 = \theta/(3\hat C)$ and to use that, since $\lambda\leqslant \lambda_0$ and $h\leqslant h_0
$,
\[\max\po Ch , \frac{\lambda}2(1+\alpha)\po \frac{16}{3c_0^3 N}  + \frac{2}{3}\pf  \pf \leqslant \frac{\theta}3\]
to conclude the proof of Proposition~\ref{prop:LyapunovRd}.

  \end{proof}

\subsection*{Acknowledgements}

The research of P.M. is supported by the projects SWIDIMS (ANR-20-CE40-0022) and CONVIVIALITY (ANR-23-CE40-0003) of the French National Research Agency. 

\bibliographystyle{plain}
\bibliography{biblio}  

\appendix
\section{Proof of auxiliary lemma}  \label{sec:auxiliaryproof}
\begin{proof}[Proof of Lemma~\ref{lem:regularity}]
    By \eqref{eq:nablaUN} and Assumption~\ref{assu:regularity}, $L_1$-Lipschitz continuity of $\nabla U_N$ follows directly, i.e., for all $\bfx,\bfy\in \X^N$
    \begin{align*}
        |\nabla U_N(\bfx) & - \nabla U_N(\bfy)|^2 = \sum_{i=1}^N |DF(\pi_{\bfx},x_i)-DF (\pi_{\bfy},y_i)|^2 
        \\ & \le \sum_{i=1}^N \po M_{1,x} |x_i-y_i| + M_{1,m} \mathcal{W}_2(\pi_{\bfx},\pi_{\bfy})\pf)^2
        \\ & \le \sum_{i=1}^N (M_{1,x}^2+M_{1,m}^2)| x_i-y_i| ^2+ 2 M_{1,x}M_{1,m} \sum_{i=1}^N |x_i-y_i|\po \frac{1}{N}\sum_{k=1}^N (x_k-y_k)^2\pf^{1/2}
        \\& \le (M_{1,x}+M_{2,m})^2|\bfx-\bfy|^2.
    \end{align*}
    By the fundamental theorem of calculus, 
    \begin{align*}
        |(\na^2 U_N(\bfx+\bfy)-\na^2 U_N(\bfx)) \bfz|^2 & = \sum_{i=1}^N \left( \sum_{j=1}^N  \int_0^1 \frac{\dd}{\dd s}\na^2_{x_i,x_j} U_N(\bfx+s\bfy) z_j \dd s \right)^2
        \\ & = \sum_{i=1}^N \left( \int_0^1 \sum_{j=1}^N   \sum_{k=1}^N \na^3_{x_i,x_j,x_k} U_N(\bfx+s\bfy) y_k z_j \dd s \right)^2. 
    \end{align*}
    For any $\bfx\in \X^N$ and $i\in[[1,N]]$,
    \begin{align*}
         \sum_{j,k=1}^N \na_{x_i,x_j,x_k}^3 U_N(\bfx) y_k z_j &  = \frac{1}{N^2} \sum_{j,k=1}^N D^3 F(\pi_{\bfx},x_i, x_j, x_k)y_k z_j+ \frac{1}{N}\sum_{j=1}^N \na_{x_j} D^2 F(\pi_{\bfx}, x_i, x_j) y_j z_j
         \\ & + \frac{1}{N}\sum_{j=1}^N \na_{x_i}D^2 F(\pi_{\bfx} x_i, x_j) y_j z_i 
          + \frac{1}{N}\sum_{k=1}^N \na_{x_i}D^2 F(\pi_{\bfx}, x_i, x_k) y_i z_k 
          \\ & + \na_{x_i}^2 D F(\pi_{\bfx}, x_i) y_i z_i \,. 
    \end{align*}
    Then by Assumption~\ref{assu:regularity}, there exists a constant $\mathbf{C}<\infty$ depending on the uniform bound of $D^3 F$, $\na D^2F$ and $\na^2 D F$ such that
    \begin{align*}
         |(\na^2 U_N(\bfx+\bfy)-\na^2 U_N(\bfx)) \bfz|^2 & \le \mathbf{C} \sum_{i=1}^N \Bigg( 
         \po \frac{1}{N^2}\sum_{j,k=1}^N y_kz_j\pf^2 + \po \frac{1}{N}\sum_{j=1}^N y_j z_j\pf^2 +\po \frac{1}{N}\sum_{j=1}^N y_j z_i\pf^2 
         \\ &  +\po \frac{1}{N}\sum_{j=1}^N y_i z_j\pf^2 +\po y_i z_i\pf^2\Bigg) 
         \\ & \le 5 \mathbf{C} \|\bfy\|_4^2 \|\bfz\|_4^2 \,,
    \end{align*}
    by Cauchy-Schwarz and since 
    \begin{align*}
        \frac{1}{N}\sum_{i=1}^N |y_i|^2 \sum_{k=1}^N|z_k|^2 \le \sqrt{\sum_{i=1}^N |y_i|^4 \sum_{k=1}^N |z_k|^4} \, . 
    \end{align*}
    
    Hence, we observe
    \begin{align*}
        |(\na^2 U_N(\bfx+\bfy)-\na^2 U_N(\bfx)) \bfz|\le \sqrt{16M_{2,x}^2+10 M_{2,m}^2}  \|\bfy\|_4\|\bfz\|_4.
    \end{align*}
    To prove \eqref{eqdef:assuU4bounded}, we observe
    \begin{align*}
        \Big|\na U_N(\bfx+\bfy) & -\nabla U_N(\bfx)-\frac{1}{2}(\na^2 U_N(\bfx) +\na^2 U_N(\bfx+\bfy))\bfy \Big|^2
        \\ & = \sum_{i=1}^N \Big| \sum_{j=1}^N \po \int_0^1 \na_{x_i,x_j}^2 U_N(\bfx+s\bfy) \dd s - \frac{1}{2}(\na^2_{x_i,x_j} U_N(\bfx) +\na^2_{x_i,x_j} U_N(\bfx+\bfy)) \pf y_j \Big|^2.
    \end{align*}
Since by integration by parts for all $f\in \mathcal{C}^2([0,1], \mathbb{R})$,
\begin{align*}
\int_0^1 f(s) \dd s - \frac{f(0)+f(1)}{2}= \int_0^1 f'(s)\po\frac{1}{2}-s\pf \dd s = \int_0^1 f''(s) \frac{s(s-1)}{2} \dd s,
\end{align*} 
it holds
\begin{align*}
\Big| \na U_N(\bfx &+ \bfy)- \na_N (\bfx)- \frac{1}{2}\na^2 U_N (\bfx)+\na^2 U_N(\bfx+\bfy) \bfy \Big|^2
\\ & = \sum_{i=1}^N \Big|\sum_{j=1}^N \int_0^1 \frac{\dd^2}{\dd s^2} \na_{x_i,x_j}^4 U_N(\bfx+ s\bfy) \frac{s(s-1)}{2}\dd s y_j\Big|^2
\\ & = \sum_{i=1}^N \Big|\sum_{j,k,l=1}^N \int_0^1 \na_{x_i,x_j,x_k,x_l}^4 U_N(\bfx+ s\bfy) \frac{s(s-1)}{2}\dd s y_j y_k y_l\Big|^2\, . 
\end{align*}
For fixed $i\in[[1, N]]$ and any $\bfz\in \mathbb{X}^N$
\begin{align*}
\sum_{j,k,l=1}^N & \na_{x_i,x_j,x_k,x_l}^4 U_N(\bfz)y_j y_k y_l = \sum_{j,k,l=1}^N \frac{1}{N^3} D^4 F(\pi_{\bfz},z_i, z_j,z_k, z_l) y_jy_k y_l 
\\ & +  \sum_{j,k=1}^N \frac{3}{N^2} \po \nabla_{x_i} D^3 F(\pi_{\bfz},z_i, z_j,z_k ) y_i y_jy_k+\nabla_{x_j} D^3 F(\pi_{\bfz},z_i, z_j,z_k ) y_j y_j y_k  \pf 
\\ & + \sum_{j=1}^N \frac{3}{N} \po \nabla_{x_i}^2 D^2 F(\pi_{\bfz},z_i, z_j) y_i^2 y_j+\nabla_{x_i,x_j}^2 D^2 F(\pi_{\bfz},z_i, z_j) y_i y_j^2 \pf 
\\ &+ \sum_{j=1}^N \frac{1}{N} \nabla_{x_j}^2 D^2 F(\pi_{\bfz},z_i, z_j) y_j^3  + \nabla_{x_i}^3 DF(\pi_{\bfz}, z_i) y_i^3 \, . 
\end{align*}
Since $|\int_0^1 \frac{s(s-1)}{2}\dd s| = \frac{1}{12}$ and since by Assumption~\ref{assu:regularity}the derivatives up to forth order of $F$ are bounded, there exists a constant $\mathbf{L}<\infty$ depending on the uniform bound of $D^4 F$, $\na D^3 F$, $\na^2 D^2 F$ and $\na^3 D F$ 
given in Assumption~\ref{assu:regularity} such that
\begin{align*}
\sum_{i=1}^N &\Big|\sum_{j,k,l=1}^N \int_0^1 \na_{x_i,x_j,x_k,x_l}^4 U_N(\bfx+ s\bfy) \frac{s(s-1)}{2}\dd s y_j y_k y_l\Big|^2 
\\ & \le \mathbf{L}
\sum_{i=1}^N \Big|\sum_{j,k,l=1}^N \frac{1}{N^3}y_j y_k y_l+\sum_{j,k=1}^N \frac{3}{N^2}(y_i y_j y_k + y_j^2 y_k) +\sum_{j=1}^N \frac{1}{N}(3 y_i^2 y_j + 3 y_i y_j^2 + y_j^3) +  y_i^3\Big|^2 
\\ & \le 7 \mathbf{L}\sum_{i=1}^N \left( \Big|\sum_{j,k,l=1}^N \frac{1}{N^3}y_j y_k y_l\Big|^2+\Big|\sum_{j,k=1}^N \frac{3}{N^2}y_i y_j y_k\Big|^2 +\Big| \sum_{j,k=1}^N \frac{3}{N^2}y_j^2 y_k\Big|^2 +\Big|\sum_{j=1}^N \frac{3}{N} y_i^2 y_j\Big|^2 \right.
\\ & \left. + \Big|\sum_{j=1}^N \frac{3}{N} y_i y_j^2\Big|^2 + \Big|\sum_{j=1}^N \frac{1}{N}y_j^3\Big|^2 +  |y_i|^6 \right) \, . 
\end{align*}
Since 
\begin{align*}
&\frac{1}{N^5}\Big|\sum_{j=1}^N y_j\Big|^6 \le \frac{1}{N^2}\Big|\sum_{j=1}^N |y_j|^2\Big|^3\le \frac{1}{N}\sum_{j=1}^N|y_j|^4 \sum_{j=1}^N |y_j|^2 \le 2\|\bfy\|_6^6 , 
\\ & \frac{1}{N^2}\Big|\sum_{i=1}^N y_i^4\Big|\Big|\sum_{j=1}^N y_j\Big|^2\le \frac{1}{N}\sum_{j=1}^N|y_j|^4 \sum_{j=1}^N |y_j|^2 \le 2\|\bfy\|_6^6, \qquad  \text{and}
\\ & \frac{1}{N} \Big|\sum_{j=1}^N y_j^3\Big|^2\le \|\bfy\|^2, 
\end{align*}
there exists a constant $L_3$ such that 
\begin{align*}
\Big| \na U_N(\bfx+\bfy)-\na U_N(\bfx)-\frac{1}{2}\po \na^2 U_N(\bfx)+\na^2 U_N (\bfx+\bfy)\pf \bfy\Big|^2\le L_3^2 \|\bfy\|_6^6\, .
\end{align*}

\end{proof}

\end{document}